\documentclass[11pt,twoside]{article}
\usepackage{amsfonts}
\usepackage{amsmath}
\usepackage{amssymb}
\usepackage{multicol}
\usepackage{graphics}
\usepackage{stmaryrd}
\usepackage{cite}

\usepackage{indentfirst}
\newtheorem{theorem}{Theorem}[section]
\newtheorem{lemma}[theorem]{Lemma}

\newtheorem{definition}[theorem]{Definition}
\newtheorem{remark}[theorem]{Remark}
\newtheorem{proposition}[theorem]{Proposition}
\numberwithin{equation}{section}
\newenvironment{proof}[1][Proof]{\noindent\textbf{#1.} }{\hfill $\Box$}
\allowdisplaybreaks

 \makeatletter
\begin{document}

\author{Ya-Hui Wang and Zhi-Cheng Wang \thanks{Corresponding author (E-mail: wangzhch@lzu.edu.cn).}  \\
{School Mathematics and Statistics, Lanzhou University}\\
{Lanzhou, Gansu 730000, People's Republic of China.}}
\title{\textbf{ Propagating Terrace and Asymptotic Profile to Time-Periodic Reaction-Diffusion Equations}} \maketitle

\begin{abstract}
This paper is concerned with the asymptotic behavior of  solutions of time periodic reaction-diffusion equation
\begin{equation*}\label{aaa}
\begin{cases}
u_{t}(x,t)=u_{xx}(x,t)+f(t,u(x,t)),\quad \,\,\forall x\in\mathbb{R},\,t>0,\\
u(x,0)=u_{0}(x), \quad \quad\quad\quad\quad\quad\quad\quad\quad \forall x\in\mathbb{R},
\end{cases}
\end{equation*}
where $u_{0}(x)$ is the Heaviside type initial function and $f(t,u)$ satisfies $f(T+t,u)=f(t,u)$. Under certain conditions,  we  prove that there exists a minimal propagating terrace (a family of pulsating traveling fronts) in some specific sense and the solution of the above equation converges to the minimal propagating terrace.


\end{abstract}

\newpage

\section{Introduction}
In this paper we investigate the long-time behavior of solutions of the following reaction-diffusion equation
\begin{equation}\label{s11}
u_{t}(x,t)=u_{xx}(x,t)+f(t,u(x,t)),\quad\,\, \forall (x,t)\in\mathbb{R}\times (0,\infty)
\end{equation}
with the initial value
\begin{equation}\label{s12}
u(x,0)=u_{0}(x), \quad\forall x\in\mathbb{R}.
\end{equation}
The nonlinearity $f\in C^2(\mathbb{R}^2,\Bbb{R}) $ satisfies
\[
f(t+T,u)=f(t,u) \,\,\,\text{and}\,\,\, f(t,0)\equiv0,\quad \forall x\in\mathbb{R},t\in\mathbb{R},
\]
where $T>0$ is a constant. Obviously, the corresponding ordinary differential equation of \eqref{s11} is
\begin{equation}\label{s13}
\begin{cases}
\omega'=f(t,\omega(t)),\quad t\in\mathbb{R},\\
\omega(0)=\beta\in\mathbb{R}.
\end{cases}
\end{equation}
Let $\omega(\beta,t)$ be a solution of \eqref{s13} relying on the initial value $\beta$. Define $P(\cdot):=\omega(\cdot,T):\Bbb{R}\to\Bbb{R}$, which is usually called the Poincar${\rm\acute{e}}$  or periodic map of \eqref{s13}. It is well-known that if $\beta_{0}$ is a fixed point of the map $P$, then $\omega(\beta_{0},t)$ is a $T$-periodic solution of \eqref{s11} and \eqref{s13}. Suppose that $\omega(\beta_{0},t)$ is a $T$-periodic solution of  \eqref{s13}. Recall  \cite{PeterHess,ZXQ2003}. We say that $\omega(\beta_{0},t)$ is isolated from below (resp.  above) with respect to \eqref{s13} if there exists no sequence of $T$-periodic solutions of \eqref{s13} converging to $\omega(\beta_{0},t)$ from below (resp. above). $\omega(\beta_{0},t)$ is said to be stable from below (resp.  above) with respect to \eqref{s13} if it is stable in the $L^{\infty}$ topology under nonpositive (resp. above) perturbations of the initial vaule around $\beta_{0}=\omega(\beta_{0},0)$.  Otherwise, $\omega(\beta_{0},t)$ is called  unstable from below (resp.  above). Suppose further that $\omega(\beta_{0},t)$ is isolated from below. Then it follows from  the Dancer-Hess connecting orbit lemma \cite{PeterHess,ZXQ2003} that  $\omega(\beta_{0},t)$ is stable from below if and only if there exists $\beta^{\ast}\in(0,\beta_{0})$ such that $\omega(\beta^{\ast},t+nT)$ converges to $\omega(\beta_{0},t)$ uniformly on $t\in [0,T]$ as $n\to\infty$. Similarly, $\omega(\beta_{0},t)$ is unstable from below if and only if there exists $\tilde{\beta}\in(0,\beta_{0})$ such that
$\omega(\tilde{\beta},t-nT)$ converges to $\omega(\beta_{0},t)$ uniformly on $ t\in [0,T]$ as $n\to\infty$.

Throughout this paper, we always suppose that $\alpha>0$ is a fixed point of the Poincar${\rm\acute{e}}$ map $P$. In this paper we focus on the Heaviside type initial value
\begin{equation}\label{s18}
u_{0}(x)=\alpha H(a-x),
\end{equation}
where  $a\in\mathbb{R}$ is a fixed constant, $H(x)$ is the Heaviside function, that is, $H(x)=0$ if $x<0$ and $H(x)=1$ if $x\geq0$. For any $a\in\mathbb{R}$, we always denote the solution of \eqref{s11} with initial value \eqref{s18} by $\tilde{u}(x,t;a)$. To describe the propagation and asymptotic behavior of solutions of \eqref{s11}, the so-called pulsating traveling front plays an important role. Here we recall the definition of pulsating traveling fronts of \eqref{s11} (see \cite{G2014CVPDE}).
\begin{definition}\label{def2convergence}
A pulsating traveling front of \eqref{s11} connecting two distinct $T$-periodic solutions $\omega_{1}(t)$ and $\omega_{2}(t)$ is an entire solution $u$ satisfying, for some $L\in\Bbb{R}$,
\[
u(x,t)=u(x-L,t+T),\,\,\forall x,t\in\mathbb{R},
\]
and
\[
u(\infty,\cdot)=\omega_{1}(\cdot)\,\,\text{and}\,\,u(-\infty,\cdot)=\omega_{2}(\cdot)\,\,\text{ locally\ uniformly\ on}\,\,\mathbb{R}.
\]
The ratio $c:=\frac{L}{T}$ is called the average speed (or simply the speed) of the traveling front.
\end{definition}
The above definition can be equivalently represented as  follows.
\begin{remark}\label{def3convergence}
 $u$ is a pulsating traveling front of \eqref{s11} connecting two distinct $T$-periodic solutions $\omega_{1}(t)$ and $\omega_{2}(t)$ with speed $c\in \Bbb{R}$ if and only if $\phi(\xi,t):=\phi(x-ct,t)=u(x,t)$ satisfies
\[
\phi_{t}-c\phi_{\xi}-\phi_{\xi\xi}-f(t,\phi)=0,\,\,\forall (\xi,t)\in\mathbb{R}^2,
\]
\[
\phi(\xi,t)=\phi(\xi,t+T),\,\forall (\xi,t)\in\mathbb{R}^2,
\]
\[
\phi(\infty,\cdot)=\omega_{2}(\cdot)\,\,\,\text{and}\,\,\,\phi(-\infty,\cdot)=\omega_{1}(\cdot)\,\,\text{uniformly on}\,\, \mathbb{R}.
\]
\end{remark}
When the nonlinearity is of bistable type, monostable type and ignition type respectively, the pulsating traveling fronts of \eqref{s11} and the long-time behavior of solutions of \eqref{s11} with front-like initial value 
have been extensively studied and here we would like to recall the existing results. Clearly, $0$ and $\alpha$ are assumed to be the fixed points of the map $P$. When the nonlinearity $f$ is of bistable type, that is, $f$ satisfies the following  structure hypotheses:

 {\it {\bf Bistable case}: $0$ is stable from above with respect to \eqref{s13}, the $T$-periodic solution $\omega(\alpha,t)$ is stable from below with respect to \eqref{s13} and there exists a unique fixed point $\theta$ of the map $P$ in $(0,\alpha)$,}

\noindent
it was proved by Alikakos et al. \cite{ABC1999TAMS} that \eqref{s11} admits a unique pulsating traveling front up to translation with a unique speed $c$ connecting two periodic solution $\omega_0(t)\equiv 0$ and $\omega_1(t)=\omega(\alpha,t)$ if
  the additional  non-generate condition
\begin{equation}\label{dd}
\left.\frac{{\rm d}}{{\rm d}\beta}P(\beta)\right|_{\beta=0,\alpha}<1<\left.\frac{{\rm d}}{{\rm d}\beta}P(\beta)\right|_{\beta=\theta}
\end{equation}
is imposed. Moreover, they showed that for any front-like initial value, the pulsating traveling front is globally exponentially asymptotically stable up to  shift. Here we would like to emphasize that the existence of the unique pulsating traveling front can also be obtained by the theory developed by Fang and Zhao \cite{FZ} when \eqref{dd} holds. In high dimensional space, pulsating curved fronts \cite{WW2011,SLW2012,W2015} were also established for \eqref{s11} with periodic bistable nonlinearity. For the time almost periodic bistable case, Shen \cite{Shen1999JDE1,Shen1999JDE2} introduced the definition of almost periodic traveling wave solutions and established  the existence, uniqueness and stability of the solutions. When $f$ is of ignition type, that is, $f$ satisfies

{\it {\bf Ignition case}: There exists $\theta\in (0,\alpha)$ such that all $\beta\in [0,\theta]$ are  fixed points  of the map $P$ and there is no fixed point of $P$ in $(\theta,\alpha)$. In addition, $\omega(\theta,t)$ is unstable from above with respect to \eqref{s13} and  $\omega(\alpha,t)$ is stable from below with respect to \eqref{s13},}

\noindent the existence, uniqueness and stability of pulsating traveling fronts of \eqref{s11} were established by Shen and Shen  \cite{SS2017JDE1,SS2017T}, where they considered  the more general equation \eqref{s11} in time heterogeneous media of ignition type. Other results of transition fronts in local and nonlocal diffusion equations with time nonautonomous nonlinearity can refer to \cite{SS2017DCDS,SS2017JDE2}. When $f$ is of monostable type, that is, $f$ satisfies

{\it {\bf Monostable case}: $0$ is unstable from above with respect to \eqref{s13}, the $T$-periodic solution $\omega(\alpha,t)$ is stable from below with respect to \eqref{s13}, and there is no fixed points  of the map $P$ in $(0,\alpha)$,}

\noindent the existence of pulsating traveling fronts can be obtained by  Liang et al. \cite{LYZ2006JDE,LZ2007}. In addition, for the front-like initial value with exponential decay near $0$, the asymptotic stability of pulsating traveling fronts with Fisher-KPP nonlinearity was  established by  Shen \cite{Shen2011}. In fact, Shen \cite{Shen2011} investigated  a class of almost periodic KPP-type reaction-diffusion equations, which covers the periodic case. Using the results of  \cite{SS2017JDE1,SS2017T} for ignition equations, Bo et al. \cite{B2019RWA}  proved the  existence of pulsating traveling fronts and the  spreading speed for a class of time periodic diffusion equations with degenerate monostable nonlinearity.

Observing the above results, we can find that, when the nonlinearity $f$ is of bistable, ignition and monostable type respectively,   equation \eqref{s11} admits pulsating traveling front connecting two periodic solutions $\omega_0(t)\equiv 0$ and $\omega(\alpha,t)$ and  the solutions of \eqref{s11} with front-like type initial value 
usually converge to the unique pulsating traveling front (the critical pulsating traveling front  for monostable nonlinearity). Such a conclusion is  true for the homogeneous equations \cite{AW1978,B1983,FM1977ARMA,U1978,VVV1994}, see also \cite{G2014CVPDE,H2008,HR2011} for spatial periodic equations. In this paper we are interested in the general nonlinearity $f$,  that is, $f$ is not one of bistable, ignition and monostable types.  Exactly, we want to know whether there exists pulsating traveling front of \eqref{s11} connecting two periodic solutions $\omega_0(t)\equiv 0$ and $\omega(\alpha,t)$ for a general nonlinearity $f$ and what are the asymptotic profile of solutions of \eqref{s11} with Heaviside type initial value.  In fact, Fife and McLeod \cite{FM1977ARMA} have considered three-stable homogeneous reaction-diffusion equations, namely, $f(t,u):=f(u)$ admits three stable equilibria $0<\theta<\alpha$. They  showed that in this case the  traveling front  connecting $0$ and $\alpha$  may not exist.  Hence,  to describe the long-time behavior of solutions with front-like initial value, a combination of two bistable traveling fronts is needed. Recently, Ducrot et al.\cite{DGM} made a great progress for such a problem and they showed that for the spatially periodic reaction-diffusion equations with suitable conditions, there exists a minimal propagating terrace (a family of pulsating traveling fronts) and the solution of the equation with the Heaviside type initial value converges to the minimal propagating terrace. More recently,  Polacik \cite{PSIAM2017} investigated the long-time behavior of the following equation
\begin{equation}\label{s15}
\begin{cases}
 u_{t}(x,t)=\triangle u(x,t)+f(u(x,t)), \quad x\in\mathbb{R}^{N},\,\,t>0,\\
 u(x,0)=u_{0}(x),\,\,\,\,\,\quad\quad \quad\quad\quad\quad\quad x\in\mathbb{R}^{N},
 \end{cases}
\end{equation}
where the nonlinearity  $f\in C^{1}(\mathbb{R})$ satisfies $f(\varrho)=f(0)=0$ with $\varrho>0$ and $N\geq2$, the initial value $u_{0}(x)$ satisfies
\[
0\leq u_{0}(x)\leq\alpha,\,\,\lim_{x_{1}\to-\infty}u_{0}(x_{1},x')=\varrho\,\, {\rm and}\,\,\lim_{x_{1}\to\infty}u_{0}(x_{1},x')=0 \quad {\rm uniform\ in}\  x'\in\mathbb{R}^{N-1}.
\]
He showed that the solution of \eqref{s15} approaches a planar propagating terrace or a stacked family of planar traveling fronts as $t\to\infty$. Polacik  \cite{PMAMS} further considered  the case of $N=1$ and the initial value $u_{0}(x)$ with
\[
\lim_{x\to-\infty}u_{0}(x)>0\,\,\,\text{and}\,\,\lim_{x\to\infty}u_{0}(x)=0.
\]
By employing the phase analysis, zero number argument and a geometric method involving the spatial trajectories $\{(u(x,t),(u_{x}(x,t))\,|\,x\in\mathbb{R},t>0\}$, he revealed that the graph of $u(x,t)$ is arbitrarily close to a propagating terrace, that is, a system of stacked traveling fronts at large times.  Du and Matano \cite{DM2018Pre} studied the propagation profile of solutions of \eqref{s15} for a high-dimensional case, where the  nonlinearity $f(u)$ is of multistable type:  $f(\alpha)=f(0)=0$, $f'(0)<0$ and $f'(\varrho)<0$  for some $\varrho>0$, $f$ may have finitely many nondegenerate zeros in the interval $(0, \varrho)$, the class of initial data $u_0(x)$ includes in particular those which are nonnegative and
decay to $0$ at infinity. They showed that if $u(\cdot,t)$ satisfies
\[
\left\|u(\cdot,t)-\alpha\right\|_{C_{loc}(\mathbb{R}^N)}\to0\,\,\,\text{as}\,\,\,t\to\infty,
\]
then the asymptotic behavior of $u(x,t)$ is determined by the one-dimensional propagating terrace introduced by Ducrot et al.\cite{DGM}.

Motivated by Ducrot et al.\cite{DGM}, in this paper we study the propagation and asymptotic behavior of solutions of \eqref{s11} with Heaviside type initial value under  the following hypothesis
\begin{enumerate}
\item[($\mathbf{H1}$)] There exists a solution $\underline{u}(x,t)$ of \eqref{s11} with compactly supported initial value $0\leq \underline{u}_{0}(x)<\alpha$, which converges locally uniformly to $\omega(\alpha,t)$ as $t\to\infty$,
\end{enumerate}
where $\alpha>0$ is the fixed point of the map $P$ as assumed before. As those done by Ducrot et al.\cite{DGM}, we will also show that there exists a minimal propagating terrace of \eqref{s11} and the solution of  equation of \eqref{s11} with the Heaviside type initial value converges to the minimal propagating terrace. Here we emphasize that the hypothesis ($\mathbf{H1}$)  is proper and covers an extensive variety of nonlinearity. In fact, the asymptotic behavior of solutions of \eqref{s11} with general nonlinearity and compactly supported initial value has been widely studied recently.
See Zlatos \cite{Z2006JAMS},  Du and Matano  \cite{DM2010JEMS} and Du and Polacik \cite{DP} for homogenous reaction-diffusion equation,   and Polacik \cite{P2011ARMA} and  Ding and Matano \cite{DingM2018} for the nonautonomous reaction-diffusion equation.

To state our main results, we first introduce the definition of a propagating terrace of \eqref{s11}. For the sake of simplicity, we assume that $\omega_{k}(t):=\omega(\alpha^{k},t)$ for all $0\leq k\leq N, k\in \mathbb{N}$, where $\alpha^{k}$ ($0\leq k\leq N$) is the fixed point  of the map $P$, that is, $\omega_{k}(t)$ is $T$-periodic solution  of \eqref{s11}.
\begin{definition}\label{defpropagatingterrace}
For $Q=\left\{\omega_{k}(t)_{0\leq k\leq N},U_{k}(x,t)_{1\leq k\leq N}\right\}$, if the following statements hold:
\begin{enumerate}
\item[(1)] $\omega_{k}(t)\,(0\leq k\leq N)$ is a $T$-periodic solution of \eqref{s11} with
\[\omega(\alpha,t)=\omega_{0}(t)>\omega_{1}(t)>\cdots>\omega_{N}(t)=0;\]
\item[(2)] $U_{k}(x,t)\,(1\leq k\leq N)$ is a pulsating traveling front of \eqref{s11} connecting $\omega_{k}(t)$ to $\omega_{k-1}(t)$;
\item[(3)]$c_{k}$ is the speed of \,$U_{k}$ satisfying $0\leq c_{1}\leq \cdots \leq c_{N}$,
\end{enumerate}
 then we call $Q$ a propagating terrace of \eqref{s11} connecting $0$ to $\omega(\alpha,t)$.
\end{definition}

It is clear that a propagating terrace is just a single pulsating traveling front in the case of monostable, bistable and ignition nonlinearities (see Theorem \ref{th3terrace}). We introduce a so-called ``minimal" notion for a propagating terrace.
\begin{definition}\label{def1terrace}
Let $Q=\{\omega_{k}(t)_{0\leq k\leq N},U_{k}(x,t)_{1\leq k\leq N}\}$ be a propagating terrace of \eqref{s11} connecting $0$ and $\omega(\alpha,t)$. We say that $Q$ is minimal if and only if the following statements are valid:

(1) if $P=\{\psi_{k}(t)_{0\leq k\leq N'},
V_{k}(x,t)_{1\leq k\leq N'}\}$ is an any other propagating terrace of \eqref{s11}, then
\begin{equation*}
\{\omega_{k}(t)\,|\,0\leq k\leq N\}\subset \{\psi_{k}(t)\,|\,0\leq k\leq N'\};
\end{equation*}

(2) $U_{k}$ is steeper than any other entire solution connecting $\omega_{k}(t)$ to $\omega_{k-1}(t)$.
\end{definition}

The concept of ``steeper"  will be given in next section. 
 The following assumption is also needed in the proof of the main results.
\begin{enumerate}
\item[($\mathbf{H2}$)] For  the periodic solution $\bar{\omega} (t) $  of \eqref{s13} with $\bar{\omega} (0)\in (0,\alpha)$, if it is isolated and unstable from below, then there exists a sequence $\{\bar{\phi}_{n}(t)\}$ of $T$-periodic supersolutions  of \eqref{s13} such that  $ \bar{\phi}_{n}(t)<\bar{\omega} (t)$ for any $t\in\Bbb{R}$ and $n\in\Bbb{N}$, and  $\bar{\phi}_{n}(t)$  converges to $\omega_{k}(t)$ uniformly in $t\in \mathbb{R}$ as $n\to\infty$.
\end{enumerate}
Here we would like to emphasize that  the hypothesis (H2) is not  harsh. In order to explain this fact,  we  introduce the following periodic eigenvalues problem:
\begin{equation}\label{eigenvalue1}
\begin{cases}
\bar{\varphi}^\prime(t)-f_{u}(t,\bar{\omega} (t))\bar{\varphi}(t)= \bar{\mu} \bar{\varphi}(t),\quad \quad \,\,\forall t\in\mathbb{R},\\
\bar{\varphi}(t)>0\,\,\text{and}\,\,\bar{\varphi}(t)=\bar{\varphi}(t+T),\quad\quad\quad\,\forall t\in\mathbb{R}.
\end{cases}
\end{equation}
A direct calculation shows that $\bar{\mu} :=-\frac{1}{T}\int_{0}^{T}f_{u}(t,\bar{\omega} (t))dt$  is an eigenvalue of problem \eqref{eigenvalue1} with the corresponding eigenfunction
\begin{equation}\label{s43}
\bar{\varphi} (t)=\exp\left(\int_{0}^{t}f_{u}(s,\bar{\omega}(s))ds-\frac{t}{T}\int_{0}^{T}f_{u}(t,\bar{\omega} (t))dt \right).
\end{equation}
Since $\bar{\omega} (t)$ is  unstable from below, then we have $\bar{\mu}\leq 0$. In this case we can show that if either $\bar{\mu}<0$ or $\bar{\mu}=0$ and there exists $\sigma>1$, $\tau>0$ and $\bar{\delta}>0$ such that
 \[
 f(t,\bar{\omega}(t)-\delta)-f(t-\bar{\omega}(t))\leq -f_u(t,\bar{\omega}(t))\delta-\tau\delta^\sigma\quad \delta\in (0,\bar{\delta}),
 \]
then the hypothesis (H2) holds true. Set $\bar{\phi}_{n}(t):=\bar{\omega}(t)-\frac{\varepsilon_0}{n}\bar{\varphi}(t)$, where  $ \varepsilon_0$ is a positive constant.  Since $f\in C^2(\Bbb{R}^2,\Bbb{R})$, we have
\begin{eqnarray*}
&&\bar{\phi}_{n}^\prime(t)-f(t,\bar{\phi}_{n}(t))\\
&=&\bar{\omega}^\prime(t)-\frac{\varepsilon_0}{n}\bar{\varphi}^\prime(t)- f\left(t,\bar{\omega}(t)-\frac{\varepsilon_0}{n}\bar{\varphi}(t)\right)\\
&=&f(t,\bar{\omega}(t))-f\left(t,\bar{\omega}(t)-\frac{\varepsilon_0}{n}\bar{\varphi}(t)\right)
-\frac{\varepsilon_0}{n}f_{u}(t,\bar{\omega} (t))\bar{\varphi}(t)-\frac{\varepsilon_0}{n}\bar{\mu} \bar{\varphi}(t)\\
&>0&
\end{eqnarray*}
for any $t\in\Bbb{R}$ and $n\in\Bbb{N}$, if we take $\varepsilon_0>0$  small enough. This implies that for any $n\in\Bbb{N}$, $\bar{\phi}_{n}(t)$ is a supersolution of \eqref{s13}. Clearly, $\bar{\phi}_{n}(t)$ satisfies all other conditions of the assumption (H2). Thus, (H2) holds.

Let us now state the first  result of this paper.
\begin{theorem}\label{th1terrace} Assume that {\rm(H1)} and {\rm(H2)} hold.
There exists a  propagating terrace
\[
Q=\{\omega_{k}(t)_{0\leq k\leq N},U_{k}(x,t)_{1\leq k\leq N}\}
 \]
 of \eqref{s11} connecting $0$ to $\omega(\alpha,t)$, which is minimal in the sense of Definition {\rm\ref{def1terrace}}. Such a minimal propagating terrace is unique in the sense that any minimal propagating terrace has the same $\omega_{k}(t)_{0\leq k\leq N}$ and that $U_{k}(x,t)$ is  unique  up to spatially translations for each $k$. Moreover, it satisfies

(1) $\omega_{k}(t)$ $(0\leq k<N)$ is isolated and stable from below with respect to \eqref{s13};

(2) $\omega_{k}(t)$ and $U_{k}(x,t)$ are steeper than any other entire solution of \eqref{s11}.
\end{theorem}

The following second result is devoted to the convergence of solution.
\begin{theorem}\label{th2terrace} Assume that {\rm(H1)} and {\rm (H2)} hold.
For any $a\in\Bbb{R}$, the solution $\tilde{u}(x,t;a)$ of \eqref{s11} converges to the minimal propagating terrace $Q=\{\omega_{k}(t)_{0\leq k\leq N},\,U_{k}(x,t)_{1\leq k\leq N}\}$ in the following sense:

(1) there exist functions $\{g_{k}(t)\}_{1\leq k\leq N}$ with $g_{k}(t)=o(t)\,(t\to\infty)$ such that
\[\tilde{u}(x+c_{k}(t-g_{k}(t)),t;a)-U_{k}(x+c_{k}t-a,t)\to0\,\,\,\text{as}\,\,\,t\to\infty\]
locally uniformly on $x\in\mathbb{R}$, $c_{k}$ is the speed of \,$U_{k}$;

(2) for any $\varepsilon>0$, there exist $C>0,K\in\mathbb{N}$ such that for any $t\geq KT,x\in I_{k,c}(t)$,
\[\|\tilde{u}(\cdot,t;a)-\omega_{k}(t)\|_{L^{\infty}({I_{k,c}(t)})}\leq\varepsilon,\]
where
\[I_{k,c}(t)=[c_{k}(t-g_{k}(t))+C,\,c_{k+1}(t-g_{k+1}(t))-C],\quad1\leq k\leq N-1,\]
\[I_{0,c}(t)=(-\infty,\,c_{1}(t-g_{1}(t))+C],\quad I_{N,c}(t)=[c_{N}(t-g_{N}(t))+C,\,\infty).\]
\end{theorem}

\begin{remark}{\rm
Assume that $Q=\{\omega_{k}(t)_{0\leq k\leq N},U_{k}(x,t)_{1\leq k\leq N}\}$ is the minimal propagating terrace established in Theorems \ref{th1terrace} and \ref{th2terrace}. By Theorem  \ref{th2terrace} (1), it is easy to find that if $c_k=c_{k+1}$, then there must be $g_k(t)-g_{k+1}(t)\to +\infty$ as $t\to\infty$.
}
\end{remark}

To consider the special case covering monostable, bistable and ignition nonlinearities, we list the following hypothesis
\begin{enumerate}
\item[($\mathbf{F}$)] There exists no fixed point $ \gamma\in (0,\alpha)$ of the map $P$ such that the periodic solution $\omega(\gamma,t)$ of \eqref{s11}  is isolated from below and stable from below.
\end{enumerate}
\begin{theorem}\label{th3terrace} Assume that {\rm(H1)}, {\rm(H2)} and {\rm(F)} hold.
 There exists a pulsating traveling front $\tilde{U}(x,t)$  connecting $0$  and $\omega(\alpha,t)$  with speed $\tilde{c}>0$, which is steeper than any other entire solution between $0$ and $\omega(\alpha,t)$. Furthermore, for any $a\in\mathbb{R}$, there exists a function $\tilde{g}(t)$ with $\tilde{g}(t)=o(t)\,(t\to\infty)$ such that
 \[
\lim_{t\to\infty}\left\|\tilde{u}(\cdot+\tilde{c}(t-\tilde{g}(t)),t;a)-\tilde{U}_{k}(\cdot+\tilde{c}t-a,t)\right\|_{L^\infty(\Bbb{R})}=0.
\]
\end{theorem}
\begin{remark}\label{remark01}{\rm
Clearly, Theorem \ref{th3terrace} implies that if $f$ is one of bistable, monostable and ignition types, and satisfies (H1) and (H2), then there exists a single pulsating traveling front connecting $0$  and $\omega(\alpha,t)$ and the solution of \eqref{s11} with Heaviside type initial value always converges to the pulsating traveling front. In contrast to the existing results, our results in Theorem \ref{th3terrace} are more general and deal with some more difficult cases. For example, if  the nonlinearity $f$ is of bistable type but not satisfies the non-degenerate condition \eqref{dd}, we can get the  existence of bistable pulsating traveling front by Theorem \ref{th3terrace} if the assumptions (H1) and (H2) are further satisfied. However,  we can not get the  existence of bistable pulsating traveling front by the theories of Alikaos et al. \cite{ABC1999TAMS} and Fang and Zhao \cite{FZ}, respectively, because they required the non-degenerate condition \eqref{dd}. Let us consider \eqref{s11} with the following $f(t,u)$:
\begin{equation}\label{aaa0603}
f(t,u):=\begin{cases}
          \begin{aligned}
          (2+\sin t)u^2(u-1)^5(4-u) ,&\quad  u<1,\\
          (4+\sin t)u^2(u-1)^5(4-u) ,&\quad  u\geq 1.
          \end{aligned}
         \end{cases}
\end{equation}
It is clear that $f(t,u)\in C^2(\Bbb{R}^2)$ and the map $P$ exactly admits three isolated fixed points $0,\ 1,\ 4$. Obviously,  $0$ and $4$ are stable and $1$ is unstable, and hence, the hypothesis (F) is satisfied. It is not difficult to verify that  $f(t,u)\geq g(u):=3u^2(u-1)^5(4-u)$  for all $ u\in[0,4]$ and $t\in\Bbb{R}$. Note that $\int_{0}^{4}3u^2(u-1)^5(4-u)du>0$. Then it follows from  Du and Polacik  \cite{DP} that for any $\theta\in (1,4)$, there exists $R_\theta>0$ such that $w(x,t)$ converges locally uniformly in $x\in\Bbb{R}$ to $4$ as $t\to\infty$, where $w(x,t)$ is the solution of the following equation
\[
 \begin{cases}
          \begin{aligned}
          w_t=w_{xx}+g(w), &\quad  t>0,\\
           w(x,0):=w_0(x) ,&\quad  x\in\Bbb{R},
          \end{aligned}
         \end{cases}
\]
the initial value $w(x,0):=w_0(x)$ satisfies $w_0(x)=\theta$ for $|x|\leq R_\theta$ and $w_0(x)=0$ for $|x|> R_\theta$. Consequently, applying the comparison principle yields that the assumption (H1) holds. Furthermore, it is easy to show that (H2) holds.   Now applying Theorem \ref{th3terrace}, we have that there exists a pulsating traveling front $\tilde{U}(x,t)$ connecting two equilibria $0$ and $4$. In this case, since the equilibria $0$ and $4$ are degenerate, we can not get the existence of pulsating traveling fronts by using  the theories of Alikaos et al. \cite{ABC1999TAMS} and Fang and Zhao \cite{FZ}.}
\end{remark}

Here we also would like to give an example of a mixed nonlinearity which is not one of bistable, ignition and monostable types. Let
\[
f(t,u):=\begin{cases}
          \begin{aligned}
           \varepsilon \rho ue^{-\rho u}(u-1)(u-3)^3,&\quad  u<3,\\
          (4+\sin t)(u-3)^3(8-u),&\quad  u\geq 3,
          \end{aligned}
         \end{cases}
\]
where $\varepsilon\in (0,1]$ and $\rho>1$. For such a nonlinearity, the equation admits four isolated equilibria $u=0,1,3,8$. In particular, the nonlinearity  is monostable on $[0,1]$ and bistable on $[1,8]$ respectively. Thus, the nonlinearity is a mixed nonlinearity. Similar to the argument in Remark \ref{remark01}, it is not difficult to verify that (H1) and (H2) hold. Since $ f_u(t,0)=27\varepsilon \rho $ and $f(t,u):=\varepsilon \rho ue^{-\rho u}(u-1)(u-3)^3\leq 27\varepsilon \rho u$ on $[0,1]$, the monostable equation on $[0,1]$ admits a family of traveling wave fronts with speed $c\geq c_2^*$, where $c_2^*=2\sqrt{27\varepsilon \rho}$ is the minimal wave speed. At the same time,  by  Theorem  \ref{th3terrace},  the  bistable equation on $[1,8]$ admits a traveling wave front with speed $ c_1^*>0$. Since the equilibria $1$ and $8$ are non-degenerate, we can show that the bistable pulsating traveling front connecting $0$ and $8$ is unique up to spatial shift by the method similar to those in \cite{ABC1999TAMS}.  In particular, the wave speed $ c_1^*>0$ is also unique. It is not difficult to find that  there holds $c_1^*>c_2^*>0$ when $\varepsilon \rho\to 0$, and $0<c_1^*<c_2^* $ when $\varepsilon \rho\to \infty$. According to Theorems \ref{th1terrace} and \ref{th2terrace}, when $c_1^*>c_2^*>0$, there exists a pulsating traveling front connecting $0$ and $8$, and the solution of the equation with Heaviside type initial value converges to the single pulsating traveling front. When $0<c_1^*<c_2^* $, there exists no pulsating traveling front connecting $0$ and $8$, and the solution of the equation with Heaviside type initial value converges to the propagating terrace, which consists of  the critical monostable  traveling front with speed $c^*_2$ connecting $0$ and $1$, and the bistable pulsating traveling front with speed $c^*_1$ connecting $1$ and $8$.

The organization of the paper is as follows. In Section $2$, we show some important preliminaries of this paper including several properties of zero-number, the estimate of spreading speed of solutions of \eqref{s11} and a key lemma (Lemma \ref{lem-limitset}) on $\omega$-limit set. Section $3$ is devoted to investigated the existence and convergence of pulsating traveling fronts connecting any two of $T$-periodic solutions around a given level set. The existence and convergence of a propagating terrace (Theorems \ref{th1terrace} and \ref{th2terrace}) are proved in Section $4$.
\section{Preliminaries}

 In this  section, we present some necessary preliminaries. In Subsection 2.1 we give the definition and properties of zero-number. In Subsection 2.2 we give the definition of the $\omega$-limit set of the solution of \eqref{s11} and \eqref{s12} and establish a key lemma on the $\omega$-limit set. In Subsection 3.3 we give  estimates to the spreading speed of the solution of \eqref{s11}.

\subsection{Definition and properties of Zero-number}
We first give the definition of  zero-number to a real-valued function, which can be found in \cite{SA,DGM,DingM2018}.
\begin{definition}\label{zn}
{\rm (see \cite[Definition 2.1]{DGM})} Let $v$ be any real-value function on $\mathbb{R}$.
 \begin{itemize}
   \item[(1)] If there exist $k\in\mathbb{N}^{\ast}$ and a sequence $\left\{x_{i}\right\}_{i=1}^{k+1}$ such that
$v(x_{i})\cdot v(x_{i+1})<0$ for $i=1,\cdots,k$,
then the supremum of all the $k$ is called the zero-number of $v$, denoted by $Z[v(\,\cdot\,)]$. When such a $k$ does not exist, namely, $v$ does not change sign, we define $Z[v(\cdot)]=0$ if $v\not\equiv0$, and $Z[0]=-1$ if $v\equiv0$.
   \item[(2)] If $Z[v(\cdot)]<\infty$, then we define a word $SGN[v(\cdot)]$ by
\begin{equation*}
SGN[v(\cdot)]=\left[\,sgn(x_{1}),sgn(x_{2}),\cdots\,,sgn(x_{k+1})\,\right],
\end{equation*}
where  $x_{1}<x_2<\cdots<x_{k+1}$ is the sequence that appears in the definition of $Z[v]$ with maximal $k$. When such a $k$ does not exist, we set $SGN[v(\cdot)]=sgn(v(x))$  if $v\not\equiv0$ and $v(x)\neq0$, and $SGN [0]=[ \,\, ]$, the empty word.
 \end{itemize}
\end{definition}

It follows from Definition \ref{zn}  that $Z[v(\cdot)]$ is the number of sign changes of $v$. For a smooth function $v$ having only simple zeros on $\Bbb{R}$,  $Z[v]$ coincides with the number of zeros. In particular, the length of the word $SGN[v]$ is equal to $Z[v]+1$. For two any words $A$ and $B$  consisting of $+$ and $-$, we denote $A\rhd B$ (or, $B\lhd A$) if $B$ is a subword of $A$. Here we refer to \cite{DGM} for more details and examples.

The following two lemmas show some basic properties of $Z[\,\cdot\,]$ and $SGN[\,\cdot\,]$, which come from Lemmas 2.2 and 2.3 of Ducrot et al. \cite{DGM}.

\begin{lemma}\label{lem-zn1}
{\rm (see \cite[Lemma 2.2]{DGM} )} Let $v(x,t)\not\equiv 0$ be a bounded solution of the following equation
\[
v_{t}(x,t)=v_{xx}(x,t)+c(x,t)v(x,t),\quad  x\in\mathbb{R},\,t\in(t_{1},t_{2}),
\]
where the coefficient function $c$ is bounded. Then, for each $t\in (t_1,t_2)$, the zeros of $v(\cdot,t)$ do not accumulate in $\Bbb{R}$. Furthermore,
\begin{itemize}
\item[(1)] $Z[v(\cdot,t)]$ and $SGN[v(\cdot,t)]$ are nonincreasing in $t$, namely, for any $t'>t''$, there hold $Z[v(\cdot,t')]\leq Z[v(\cdot,t'')]$ and $SGN[v(\cdot,t')]\lhd SGN[v(\cdot,t'')]$.
Here the assertion remains true even for $t''=t_{1}$ if $v$ can be extended to a continuous function on $\mathbb{R}\times[t_{1},t_{2})$.

\item[(2)] if there exist $x'\in\mathbb{R}$ and $t'\in(t_{1},\,t_{2})$ such that $v(x',\,t')=v_{x}(x',\,t')=0$, then
$Z[v(\cdot,t)]-2\geq Z[v(\cdot,s)]\geq 0$ for any $t\in(t_{1},t')$ and $s\in(t',t_{2})$
whenever $Z[v(\cdot,t)]<\infty$.
\end{itemize}
\end{lemma}
\begin{lemma}\label{lem-zn2}
{\rm (see \cite[Lemma 2.3]{DGM})} If $\{v_{n}\}_{n\in\mathbb{N}}$ is a sequence of functions converging to $v$ pointwise on $\mathbb{R}$, then
\[
Z[v]\leq\liminf_{n\to\infty}Z[v_{n}],\quad SGN[v]\lhd\liminf_{n\to\infty}SGN[v_{n}].
\]
\end{lemma}

Statement (1) of Lemma \ref{lem-zn1} follows from statement (2), while statement (2) is essentially due to \cite{SA} (see also \cite{DM2010JEMS}). Lemma \ref{lem-zn2} shows that $Z[\cdot]$ and $SGN[\cdot]$ are semi-continuous under the sense of the pointwise convergence.
Using these two lemmas, we can get the following lemma for the solutions of \eqref{s11}.  The proof is completely similar to those of \cite[Lemma 2.4]{DGM}, so we omit it.

\begin{lemma}\label{lem-zn3}
Let $u_{1}$ and $u_{2}$ be solutions of \eqref{s11} with initial values $u_1(x,0)$ and $u_2(x,0)$ respectively, where $u_{1}(x,0)$ is a piecewise continuous bounded function on $\Bbb{R}$ and $u_{2}(x,0)$ is continuous and bounded in $\mathbb{R}$. If $Z[u_{1}(\cdot,0)-u_{2}(\cdot,0)]<\infty$, then the following statements hold:

(1) for any $0\leq t<t'<\infty$, one has
\[
Z[u_{1}(\cdot,t')-u_{2}(\cdot,t')]\leq Z[u_{1}(\cdot,t)-u_{2}(\cdot,t)],
\]
\[
SGN[u_{1}(\cdot,t')-u_{2}(\cdot,t')]\lhd SGN[u_{1}(\cdot,t)-u_{2}(\cdot,t)];
\]

(2) if there exists $t'>0$ such that the graph of $u_{1}(\cdot,t')$ and that of $u_{2}(\cdot,t')$ are tangential at some point in $\mathbb{R}$, and $u_{1}\not\equiv u_{2}$, then, for any $t$ and $s$ with $0\leq t<t'<s$, one has
\[
Z[u_{1}(\cdot,t)-u_{2}(\cdot,t)]-2\geq Z[u_{1}(\cdot,s)-u_{2}(\cdot,s)]\geq0.
\]
If $u_1$ and $u_2$ are entire solutions of \eqref{s11}, then the same conclusion remains true for any $t'\in\Bbb{R}$ and $-\infty<t<t'<s<\infty$.
\end{lemma}

 Here we would like to notice that an entire solution of \eqref{s11} means a  solution defined for all $t\in\mathbb{R}$. Now we introduce a so-called ``steeper" notion, which describes the intersection of the graphs of two entire solutions of \eqref{s11}. Such a notion was first introduced by Ducrot et al. \cite[Definition 1.6]{DGM}.
\begin{definition}\label{steeper}
Let $u_{1}$ and $u_{2}$ be two entire solutions of \eqref{s11}. We say that $u_{1}$ is steeper than $u_{2}$ if  for any $t' \in\mathbb{R}$, $x_{1}\in\mathbb{R}$ and $k\in\mathbb{Z}$ such that $u_{1}(x_{1},t')=u_{2}(x_{1},t'+kT)$, we have either
$u_{1}(\cdot,\cdot+t')\equiv u_{2}(\cdot,\cdot+t'+kT)$  or $\partial _{x}u_{1}(x_{1},t')<\partial _{x}u_{2}(x_{1},t'+kT)$.
 \end{definition}

We note that Definition \ref{steeper} is slightly different from Definition 1.6 in \cite{DGM} since we are considering a time periodic equation.  The above definition implies that if  $u_{1}$ is steeper than $u_{2}$, then for any $k\in\mathbb{Z}$, either the graph of $u_1(\cdot,\cdot)$ is identical with that of $u_2(\cdot,\cdot+kT)$ or  they can intersect at most once. If their graphs never intersect, then they are steeper than each other. Based on the above definition, we have the following observation.

\begin{proposition}\label{prosteeper}
Let $u_{1}$ and $u_{2}$ be two entire solutions of \eqref{s11}. Suppose that
\[
SGN[u_{1}(\cdot,t')-u_{2}(\cdot,t'+kT)]\lhd [+\,-] \quad {\rm for\ any}\ t'\in\mathbb{R}\ {\rm and}\ k\in\mathbb{Z}.
\]
Then $u_{1}$ is steeper than $u_{2}$.
\end{proposition}
\begin{proof}
Fix any $t'\in\mathbb{R}$ and $k\in\mathbb{Z}$.
According to the assumption, we have
\begin{equation}\label{s21}
SGN[u_{1}(\cdot,t+t')-u_{2}(\cdot,t+t'+kT)]\lhd[+\,-],\quad\forall t\in\Bbb{R},
\end{equation}
which implies that
\[
Z[u_{1}(\cdot,t+t')-u_{2}(\cdot,t+t'+kT)]\leq 1,\quad\forall t\in\Bbb{R}.
\]
If $u_{1}(\cdot,\cdot+t')\not\equiv u_{2}(\cdot,\cdot+t'+kT)$, then from Lemma \ref{lem-zn3} (2), the function $u_{1}(\cdot,t')-u_{2}(\cdot,t'+kT)$ admits at most one zero on $\Bbb{R}$, and the zero must be simple. Suppose the zero is $x_1\in\Bbb{R}$. Then the simplicity of the zero, the equality $u_{1}(x_{1},t')-u_{2}(x_{1},t'+kT)=0$ and  \eqref{s21} imply $\partial_{x}u_{1}(x_{1},t')<\partial_{x}u_{2}(x_{1},t'+kT)$. By Definition \ref{steeper},  we have that $u_{1}$ is steeper than $u_{2}$. This completes the proof.
\end{proof}

\subsection{$\omega$-limit set of the solution $\tilde{u}(x,t;a)$}
\noindent

It is known that $\omega$-limit set can describe the long-time behavior of solution in fixed compact regions. In the following we give the definition of $\omega$-limit set of the solution of the Cauchy problem \eqref{s11} and \eqref{s12}, which is slightly different from the standard one. Such a  definition is useful to capture the asymptotic behavior of solution of \eqref{s11} in various moving coordinate frames, see also \cite{DGM,PMAMS,PSIAM2017}.
\begin{definition}\label{deflimits}
Let $u$ be any bounded solution of the Cauchy problem \eqref{s11} and \eqref{s12}. If there exist two sequences $\{x_{j}\}_{j\in\mathbb{N}}\subset\mathbb{R}$  and $\{k_{j}\}_{j\in\mathbb{N}}\subset\mathbb{Z}$ with $k_{j}\to\infty \,(j\to\infty)$ such that
\[
u(\cdot+x_{j},\cdot+k_{j}T)\to w(\cdot,\cdot) \quad \text{as}\ j\to\infty\quad\text{locally uniformly on}\,\,\mathbb{R}^{2},
\]
then we call $w$ an $\omega$-limit orbit of solution $u$.
\end{definition}

As mentioned by Ducrot et al. \cite{DGM}, the above convergence takes place in $C^2$ in $x$ and $C^1$ in $t$, so any $\omega$-limit orbit of solution $u$ is an entire solution of \eqref{s11}. In particular,  if $w(x,t)$ is an $\omega$-limit orbit of solution $u$, then for any $y\in\Bbb{R}$ and $k\in\Bbb{Z}$, $w(x+y,t+kT)$ is also an $\omega$-limit orbit. According such a definition, we have the following key lemma on the $\omega$-limit set of the solution $\tilde{u}(x,t;a)$. 
\begin{lemma}\label{lem-limitset}
For $a\in\mathbb{R}$. Let $w_{1}(x,t)$ be any $\omega$-limit orbit of $\tilde{u}(x,t;a)$. Then for any entire solution $w(x,t)$ of \eqref{s11} satisfying $0\leq w(x,t)\leq \omega(\alpha,t)$ for all $(x,t)\in\Bbb{R}^2$, $w_{1}(x,t)$ is steeper than $w(x,t)$ in the sense of Definition \ref{steeper}.
\end{lemma}
\begin{proof}
It follows from  Definition \ref{deflimits}, there exist  $\{x_{j}\}_{j\in\mathbb{N}}\subset\mathbb{R}$ and $\{k_{j}\}_{j\in\mathbb{N}}\subset\mathbb{Z}$ with $k_{j}\to\infty \,(j\to\infty)$ such that
\[
\tilde{u}(\cdot+x_{j},\,\cdot+k_{j}T;a)\to w_{1}(\cdot,\cdot) \quad {\rm as}\ j\to\infty \quad {\rm in}\ C_{loc}^{2,1}(\Bbb{R}^2).
\]
Let $w(x,t)$ be  any entire solution  of \eqref{s11} satisfying $0\leq w(x,t)\leq \omega(\alpha,t)$ for all $(x,t)\in\Bbb{R}^2$. Clearly, for any $i\in\mathbb{Z}$ and and $j\in\Bbb{N}$, we have $w(x-x_j,iT-k_jT )\leq \alpha=\tilde{u}(x,0;a)$ for $x\leq a$ and $w(x-x_j,iT-k_jT )\geq 0=\tilde{u}(x,0;a)$ for $x>a$, which implies that the function $\tilde{u}(\cdot,0;a)-w(\cdot-x_j,iT-k_jT )$ changes sign just once at point $x=a$. Therefore, there are
\begin{equation}\label{le1limits1}
SGN[\tilde{u}(\cdot+x_j,0;a)-w(\cdot,iT-k_{j}T)]=[+\,-]\ \ {\rm and}\  Z[\tilde{u}(\cdot+x_j,0;a)-w(\cdot ,iT-k_{j}T)]=1
\end{equation}
 for any $i\in\Bbb{Z}$ and $j\in\Bbb{N}$.
Applying Lemma \ref{lem-zn1}, for any $t\geq-k_{j}T$, one has
\[SGN[\tilde{u}(\cdot+x_j,t+k_{j}T;a)-w(\cdot,t+iT)]\lhd SGN[\tilde{u}(\cdot+x_j,0;a)-w(\cdot,iT-k_{j}T )],\]
\[Z[\tilde{u}(\cdot+x_j,t+k_{j}T;a)-w(\cdot,t+iT )]\leq Z[\tilde{u}(\cdot+x_j,0;a)-w(\cdot,iT-k_{j}T )],\]
which combining with \eqref{le1limits1} yields
\[
SGN[\tilde{u}(\cdot+x_j,t+k_{j}T;a)-w(\cdot,t+iT )]\lhd [+\,-]\ {\rm and}\ Z[\tilde{u}(\cdot+x_j,t+k_{j}T;a)-w(\cdot,t+iT)]\leq 1
\]
for any $i\in\Bbb{Z}$ and $j\in\Bbb{N}$. Letting $j\to\infty$, by Lemma \ref{lem-zn2} we obtain
\[
SGN[w_{1}(\cdot,t )-w(\cdot,t+iT )]\lhd [+\,-]\ {\rm and}\
 Z[w_{1}(\cdot,t )-w(\cdot,t+iT )]\leq 1
 \]
 for all $t\in\mathbb{R}$ and $i\in\mathbb{Z}$. Thus, applying Proposition \ref{prosteeper} gives that $w_{1}$ is steeper than $w$. This completes the proof.
\end{proof}

\subsection{Spreading speed of solution $\tilde{u}(x,t;a)$}
\noindent

 In this subsection, we establish the spreading properties of the solution $\tilde{u}(x,t;a)$ of \eqref{s11}. This result will be used repeatedly later.
\begin{lemma}\label{le1speed}
There exist $0<c_{\ast}<c^{\ast}$ such that

(1) for any $c>c^{\ast}$, $\lim_{t\to\infty}\sup_{ x\geq ct} |\tilde{u}(x,t;a)|=0$;

(2) for any $0<c<c_{\ast}$, $\lim_{t\to\infty}\sup_{x\leq ct} |\tilde{u}(x,t;a)-\omega(\alpha,t)|=0$.
\end{lemma}
\begin{proof}(1) Let $M=\max\limits_{t\in\mathbb{R}}\omega(\alpha,t)$ and $K=\max\limits_{[0,\infty)\times[0,M]}|f_{u}(t,u)|$. For all $t\in\Bbb{R}$ and $0\leq u\leq M$, we have
$|f(t,u)|\leq Ku$.
It is easy to show that the function $\bar u(x,t)=\min\left\{\omega(\alpha,t),\,\,\alpha e^{-\sqrt{K}\left(x-2\sqrt{K}t-a\right)}\right\}$
is a super-solution of \eqref{s11}.
Since $\bar{u}(x,0)\geq \tilde{u}(x,0;a)$ for any $x\in\mathbb{R}$, by the comparison principle, one has $\tilde{u}(x,t;a)\leq\bar{u}(x,t)$ for all $x\in\mathbb{R}$ and $t\geq 0$. Let $c^{\ast}=2\sqrt{K}$. Then for any $c>c^{\ast}$, we can easily show that $\lim_{t\to\infty}\sup_{ x\geq ct} |\tilde{u}(x,t;a)|=0$.

(2) Let $\underline{u}_{0}$ be the compactly supported function given in the hypothesis  (H1). It follows from the hypothesis  (H1) that the solution $\underline{u}(x,t)$ of \eqref{s11}  with initial value $\underline{u}_{0}$ converges to $\omega(\alpha,t)$ as $t\to\infty$ locally uniformly in $x\in\Bbb{R}$. Without loss of generality, we assume  that $supp(\underline{u}_{0})\subset(-\infty,a\,]$. Thus there exists $k\in\mathbb{N }$ such that
\begin{equation}\label{speed1}
\underline{u}(x,kT)\geq\max\left\{\underline{u}_{0}(x),\,\underline{u}_{0}(x-1)\right\},\,\forall x\in\mathbb{R}.
\end{equation}
Using the comparison principle, we have
\[
\underline{u}(x,t+kT)\geq\max\left\{\underline{u}(x,t),\,\underline{u}(x-1,t)\right\},\,\forall t\geq 0.
\]
Combining the above statement with \eqref{speed1}, we have 
\[\underline{u}(x,2kT)\geq\max\left\{\underline{u}_{0}(x-m)\,|\,0\leq m\leq2,m\in\mathbb{N}\right\}.\]
By the comparison principle and induction, we finally obtain that
\begin{equation}\label{speed2}
\underline{u}(x,nkT)\geq\max\left\{\underline{u}_{0}(x-m)\,|\,0\leq m\leq n,m\in\mathbb{N}\right\},\,\forall x\in\Bbb{R},\ n\in\mathbb{N}.
\end{equation}
Since $supp(\underline{u}_{0})\subset(-\infty,a]$, then $\tilde{u}(x,0;a)\geq\max\{\underline{u}_{0}(x+i)\,|\,i\in\mathbb{N}\}$. Using the comparison principle again, one has
\begin{equation}\label{speed3}
\tilde{u}(x,t;a)\geq\max\left\{\underline{u}(x+i,t)\,|\,i\in\mathbb{N}\right\},\quad \forall\ x\in\Bbb{R},\ t\geq0.
\end{equation}
Combining \eqref{speed2} with \eqref{speed3}, we get
\[
\tilde{u}(x,nkT;a)\geq\max\left\{\underline{u}_{0}(x-i)\,|\,i\leq n,\,i\in\mathbb{Z}\right\} ,\quad \forall\ x\in\Bbb{R},\  n\in\mathbb{N}.
\]
Applying the comparison principle yields
\begin{equation}\label{speed4}
\tilde{u}(x,t+nkT;a)\geq\max\left\{\underline{u}(x-i,t)\,|\,i\leq n,\,\,\,i\in\mathbb{Z}\right\},\,\forall t\geq 0,\ n\in\mathbb{N},\ x\in\mathbb{R}.
\end{equation}
Due to the convergence of $\underline{u}(x,t)$ as $t\to\infty$, here we have
\begin{equation}\label{speed5}
\max\left\{\underline{u}(x-i,t)\,|\,i\leq n,\,i\in\mathbb{Z}\right\}\to\omega(\alpha,t)\quad {\rm as}\ t\to\infty
\end{equation}
with respect to $n\in\Bbb{N}$ and $x\in(-\infty,n]$.

  Let $c_{\ast}=\frac{1}{kT}$ and $0<c<c_{\ast}$. For any $t\geq0$, we denote
$\tau(t):=t-[ct]kT$,
where $[ct]$ denotes the least integer not smaller than $ct$. It is obvious that $\tau(t)\to\infty (t\to\infty)$. In addition, there is $\omega(\alpha,t)=\omega(\alpha,\tau(t))$ due to the periodicity. Now, using \eqref{speed4}, \eqref{speed5} and the inequality $\tilde{u}(x,t;a)\leq \omega(\alpha,t)$ yield
\[
\sup_{x\leq [ct]}\left|\tilde{u}(x,\tau(t)+[ct]kT;a)-\omega(\alpha,t)\right|=\sup_{x\leq [ct]}\left|\tilde{u}(x,\tau(t)+[ct]kT;a)-\omega(\alpha,\tau(t))\right|\to0
\]
as $t\to\infty$.
Consequently, for any $t>0$,  since $ct\leq [ct]$ and $t=\tau(t)+[ct]kT$, we then have
 \[
 \lim_{t\to\infty}\sup_{x\leq ct}|\tilde{u}(x,t;a)-\omega(\alpha,t)|=0.
 \]
 This completes the proof.
\end{proof}

\section{Convergence to a pulsating traveling front for some level sets}
\noindent

 We exhibit the convergence of solutions of \eqref{s11} around a given level set in this  section. Thus we can obtain some important properties of pulsating traveling fronts. We firstly give some properties of the solution $\tilde{u}(x,t;a)$.

\begin{lemma}\label{lem3.1} The following statements hold:

(1) $0\leq \tilde{u}(x,t;a)\leq \omega(\alpha,t)$  for any $x\in\Bbb{R}$ and $t\geq 0$.

 (2) $\tilde{u}(x,t;a-x_0)=\tilde{u}(x+x_0,t;a)$ for any $x\in\Bbb{R}$, $x_0\in\Bbb{R}$, $t\geq 0$.

 (3) $\lim_{x\to\infty}\tilde{u}(x,t;a)=0$ and $\lim_{x\to-\infty}\tilde{u}(x,t;a)=\omega(\alpha,t)$ for any $t>0$, and $\partial_x\tilde{u}(x,t;a)<0$ for any $x\in\Bbb{R}$ and $t>0$.

 (4) The map $a\longmapsto \tilde{u}(x,t;a)$ is increasing, that is, if $a_1<a_2$, then
$\tilde{u}(x,t;a_1)<\tilde{u}(x,t;a_2)$ for  all $x\in\mathbb{R}$ and $t>0$.
\end{lemma}
The proof of the lemma is easy and we omit it. Following  this lemma, we have that, for any given $k\in\mathbb{N}$ and $\lambda\in(0,\alpha)$, there exists a unique $\ell(k,\lambda)\in\Bbb{R}$ such that $\tilde{u}(\ell(k,\lambda),kT;0)=\lambda$. Since the solution is  shift invariant, we then  have $\tilde{u}(0,kT;-\ell(k,\lambda))=\lambda$. It follows from Lemma \ref{lem3.1} (ii) that $-\ell(k,\lambda)$ is the unique root of the equation $\tilde{u}(0,kT;a)=\lambda$ on the variable $a$. Thus, we can give the following definition.
\begin{definition}\label{def1convergence}
For $k\in\mathbb{N}$ and $\lambda\in(0,\alpha)$, we define $\ell(k,\lambda)\in\mathbb{R}$ by
\[
\tilde{u}(0,kT;-\ell(k,\lambda))=\lambda.
\]
\end{definition}
It is clear that, for any $k\in\mathbb{N}$ and $\lambda\in(0,\alpha)$, $\ell(k,\lambda)$ is uniquely determined. In particular, there holds $\tilde{u}(-a,kT;a-\ell(k,\lambda))=\lambda$ for any $a\in\Bbb{R}$. By Lemma \ref{le1speed}, we can  easily prove an important property of  $\ell(k,\lambda)$.
\begin{proposition}\label{pro}
$\ell(k,\lambda)\to\infty$ as $k\to\infty$. In particular, there is
\begin{equation}\label{converge5}
c_{\ast}T\leq\liminf_{j\to\infty}\frac{\ell(j,\lambda)}{j}\leq
\limsup_{j\to\infty}\frac{\ell(j,\lambda)}{j}\leq c^{\ast}T.
\end{equation}
\end{proposition}

The following lemma shows the convergence of the solutions of  \eqref{s11} with some shifted Heaviside type initial value.
\begin{lemma}\label{le1convergence}
Let $\lambda\in(0,\alpha)$. Then there exists the limit
\begin{equation}\label{converge1}
\lim_{j\to\infty}\tilde{u}(\cdot,\cdot+jT;-\ell(j,\lambda)):=u_{\infty}(\cdot,\cdot;\lambda)\quad {\rm in}\ C^{2,1}_{loc}(\Bbb{R}^2),
\end{equation}
where the limit function $u_{\infty}(\cdot,\cdot;\lambda)$ satisfies:

(1) $0<u_{\infty}(x,t;\lambda)<\omega (\alpha,t)$ for all $(x,t)\in\Bbb{R}^2$;

(2) $u_{\infty}$ is an entire solution of \eqref{s11} and steeper than any other entire solution;

(3) $u_{\infty}$ is either a $T$-periodic solution of \eqref{s11} or decreasing with respect to $x$, that is $\partial_{x}u_{\infty}(x,t;\lambda)<0$ for any $(x,t)\in\mathbb{R}^{2}$.
\end{lemma}

\begin{proof}  By the standard parabolic estimates, there exists $M>0$ such that
\[
\|\tilde{u}(\cdot,\cdot;a)\|_{C^{2+\theta,1+\frac{\theta}{2}}}\left(\Bbb{R}\times [T,+\infty)\right)\leq M,
\]
 where $M$ is independent of $a\in\Bbb{R}$. Therefore, for the sequence $\{\tilde{u}(x,t+jT;-\ell(j,\lambda))\}_{j\in\Bbb{N}}$,  there exists a subsequence $\{j_{n}\}_{n\in\mathbb{N}}$ with $j_{n}\to\infty\, (n\to\infty)$ such that
\[
\lim_{n\to\infty}\tilde{u}(x,t+j_{n}T;-\ell(j_{n},\lambda))=u_{\infty}(x,t;\lambda) \quad {\rm in}\ C^{2,1}_{loc}(\Bbb{R}^2).
\]
Since $\tilde{u}(x,t+j_{n}T;-\ell(j_{n},\lambda))=\tilde{u}(x+\ell(j_{n},\lambda),t+j_{n}T;0)$, then $u_{\infty}(x,t;\lambda)$ is an $\omega$-limit orbit of $\tilde{u}(x,t;0)$. Moreover, $u_{\infty}(0,0;\lambda)=\lambda$. Since $0\leq u_{\infty}(x,t;\lambda)\leq \omega (\alpha,t)$ for any $(x,t)\in\Bbb{R}^2$, by the strong maximum principle we have $0< u_{\infty}(x,t;\lambda)<\omega (\alpha,t)$ for any $(x,t)\in\Bbb{R}^2$.

For any other sequence $\{k_{n}\}_{n\in\mathbb{N}}$ with $k_n\to \infty$ as $n\to\infty$, assume that
\[
\lim_{n\to\infty}\tilde{u}(x,t+k_{n}T;-\ell(k_{n},\lambda))=v_{\infty}(x,t;\lambda) \quad {\rm in}\ C^{2,1}_{loc}(\Bbb{R}^2).
\]
Similarly, $v_{\infty}(x,t;\lambda)$ is an $\omega$-limit orbit of $\tilde{u}(x,t;0)$ and satisfies  $v_{\infty}(0,0;\lambda)=\lambda$. By Lemma \ref{lem-limitset}, $u_{\infty}(x,t;\lambda)$ and $v_{\infty}(x,t;\lambda)$ are steeper than each other. Due to $u_{\infty}(0,0;\lambda)=v_{\infty}(0,0;\lambda)=\lambda$ and Definition \ref{steeper}, we have $u_{\infty}(x,t;\lambda)\equiv v_{\infty}(x,t;\lambda)$ in $\Bbb{R}^2$, which implies that $u_{\infty}(x,t;\lambda)$ does not depend on the choice of sequence $\{j_{n}\}_{n\in\mathbb{N}}$. Thus, we complete the proofs of \eqref{converge1} and the statement (1).

It follows from Lemma \ref{lem3.1} (2) that $\partial_{x}u_{\infty}(x,t)\leq0$ in $\Bbb{R}^2$. Then by the strong maximum principle,  we have either $\partial_{x}u_{\infty}(x,t;\lambda)\equiv0$ in $\Bbb{R}^2$ or $\partial_{x}u_{\infty}(x,t;\lambda)<0$ for any $(x,t)\in\Bbb{R}^2$. If $\partial_{x}u_{\infty}(x,t;\lambda)\equiv0$ in $\Bbb{R}^2$, then $u_{\infty}(x,t;\lambda)\equiv u_{\infty}(t;\lambda)$, which implies that $u_{\infty}(x,t;\lambda)$ is a solution independent of spatial variable. This completes the proof of the statement (2).
\end{proof}

The following part shows that the above limit $u_{\infty}$ is either a pulsating traveling front or a $T$-periodic solution of \eqref{s11}. Let us define the sequence
\[
\ell_{j}=
\begin{cases}
\ell(1,\lambda)\qquad \quad\quad\quad \quad\quad\quad\quad  j=1,\\
\ell(j,\lambda)-\ell(j-1,\lambda)\,\qquad\quad \,\,j\geq2.
\end{cases}
\]
Clearly, $\ell(j,\lambda)=\sum^{j}_{i=1}\ell_{i}$  for all $j\in\mathbb{N}$.

\begin{lemma}\label{le2convergence}
For any $\lambda\in(0,\alpha)$, the entire solution
\[u_{\infty}(x,t;\lambda)=\lim_{j\to\infty}\tilde{u}(x,t+jT;-\ell(j,\lambda))\]
of \eqref{s11}, which is defined by Lemma \rm{\ref{le1convergence}}, is either a positive $T$-periodic solution or a pulsating traveling front.
\end{lemma}
\begin{proof}
We divide the proof into two cases.

(1) If there exists a subsequence $\{\ell_{j_k}\}_{k\in\mathbb{N} } $ such that $\ell_{j_k}\to L (L >0)$ as $k\to\infty$, then for any $(x,t)\in\Bbb{R}$, one has
\begin{equation*}
\begin{aligned}
u_{\infty}(x-L,t;\lambda)
=&\lim_{k\to\infty}\tilde{u}(x-\ell_{j_k},t+j_kT;-\ell(j_k,\lambda))\\
=&\lim_{k\to\infty}\tilde{u}(x-(\ell(j_k,\lambda)-\ell(j_k-1,\lambda)),t+j_kT;-\ell(j_k,\lambda))\\
=&\lim_{k\to\infty}\tilde{u}(x,(t+T)+(j_k-1)T;-\ell(j_k-1,\lambda))\\
=&u_{\infty}(x,t+T;\lambda)
\end{aligned}
\end{equation*}
that is
\begin{equation}\label{converge3}
u_{\infty}(x-L,t;\lambda)=u_{\infty}(x,t+T;\lambda)\quad {\rm for\ any}\  (x,t)\in\Bbb{R}^2.
\end{equation}
According to Lemma \ref{le1convergence}, either $u_{\infty}(x,t;\lambda):=u_{\infty}(t;\lambda)$ is independent of $x\in\Bbb{R}$ or $\partial_xu_{\infty}(x,t;\lambda)<0$ for any $(x,t)\in\Bbb{R}$. If $u_{\infty}(x,t;\lambda):=u_{\infty}(t;\lambda)$ is independent of $x\in\Bbb{R}$, then it follows from \eqref{converge3} that $u_{\infty}(t;\lambda)=u_{\infty}(t+T;\lambda)$ for any $t\in\Bbb{R}$, which implies that $u_{\infty}(x,t;\lambda):=u_{\infty}(t;\lambda)$ is a positive $T$-periodic solution. If $\partial_xu_{\infty}(x,t;\lambda)<0$ for any $(x,t)\in\Bbb{R}$,  we have $u_{\infty}(x,t;\lambda)\to\omega_{\pm}(t)$ as $x\to\pm\infty$ and both $\omega_{+}(t)$ and $\omega_{-}(t)$ are positive $T$-periodic solutions, which implies that $u_{\infty}(x,t;\lambda) $ is a pulsating traveling front.

(2) We assume that no subsequence of $\{\ell_{j}\}_{j\in\Bbb{N}}$ converges to some positive constant. In this case we can show that $u_{\infty}$ must be a $T$-periodic solution of \eqref{s11}. Due to \eqref{converge5}, we have that there exist two subsequences $\{\ell_{j_{-,k}}\}_{k\in\Bbb{N}}$ and $\{\ell_{j_{+,k}}\}_{k\in\Bbb{N}}$ of $\{\ell_{j}\}_{j\in\Bbb{N}}$ such that
\[
\lim_{k\to\infty}\ell_{j_{+,k}}=\limsup_{j\to\infty}\ell_{j}=+\infty,\quad \lim_{k\to\infty}\ell_{j_{-,k}}=\liminf_{j\to\infty}\ell_{j}=L^-\in [-\infty,0].
\]
Now we consider two cases:

 {\bf Case 1}. $L^-\in [-\infty,0)$. In view of Lemma \ref{lem3.1} (3), for any $M\in (L^-,0)$, one has
\begin{equation*}
\begin{aligned}
&u_{\infty}(x-M,t;\lambda)\\
=&\lim_{k\to\infty}\tilde{u}(x-M,t+j_{-,k}T;-\ell(j_{-,k},\lambda))\\
\geq &\lim_{k\to\infty}\tilde{u}(x-\ell_{j_{-,k}},t+j_{-,k}T;-\ell(j_{-,k},\lambda))\\
=&\lim_{k\to\infty}\tilde{u}(x-(\ell(j_{-,k},\lambda)-\ell(j_{-,k}-1,\lambda)),t+j_{-,k}T;-\ell(j_{-,k},\lambda))\\
=&\lim_{k\to\infty}\tilde{u}(x,(t+T)+(j_{-,k}-1)T;-\ell(j_{-,k}-1,\lambda))\\
=&u_{\infty}(x,t+T;\lambda)
\end{aligned}
\end{equation*}
and
\begin{equation*}
\begin{aligned}
&u_{\infty}(x+M,t;\lambda)\\
=&\lim_{k\to\infty}\tilde{u}(x+M,t+j_{+,k}T;-\ell(j_{+,k},\lambda))\\
\leq &\lim_{k\to\infty}\tilde{u}(x-\ell_{j_{+,k}},t+j_{+,k}T;-\ell(j_{+,k},\lambda))\\
=&\lim_{k\to\infty}\tilde{u}(x-(\ell(j_{+,k},\lambda)-\ell(j_{+,k}-1,\lambda)),t+j_{+,k}T;-\ell(j_{+,k},\lambda))\\
=&\lim_{k\to\infty}\tilde{u}(x,(t+T)+(j_{+,k}-1)T;-\ell(j_{+,k}-1,\lambda))\\
=&u_{\infty}(x,t+T;\lambda)
\end{aligned}
\end{equation*}
for any $(x,t)\in\Bbb{R}^2$, which implies that
\begin{equation}\label{converge10}
 u_{\infty}(x-M,t;\lambda)\geq u_{\infty}(x,t+T;\lambda)\geq u_{\infty}(x+M,t;\lambda)\quad {\rm for\ any}\ (x,t)\in\Bbb{R}^2\ {\rm and}\ M\in (L^-,0).
\end{equation}
On the other hand, it follows from $\partial_x u_{\infty}(x,t;\lambda)\leq 0$ that
\begin{equation}\label{converge11}
 u_{\infty}(x-M,t;\lambda)\leq  u_{\infty}(x+M,t;\lambda)\quad {\rm for\ any}\ (x,t)\in\Bbb{R}^2\ {\rm and}\ M\in (L^-,0).
\end{equation}
Consequently, it follows from \eqref{converge10}, \eqref{converge11} and  the arbitrariness of $(x,t)\in\Bbb{R}^2$ and $M\in (L^-,0)$ that $u_{\infty}(x,t;\lambda)=u_{\infty}(t;\lambda)$ and $u_{\infty}(t;\lambda)=u_{\infty}(t+T;\lambda)$. That is, $u_{\infty}(x,t;\lambda)=u_{\infty}(t;\lambda)$ is a $T$-periodic solution of \eqref{s11}.

 {\bf Case 2}. $L^-=0$. Consider the the subsequences $\{\ell_{j_{-,k}}\}_{k\in\Bbb{N}}$ and $\{\ell_{j_{+,k}}\}_{k\in\Bbb{N}}$ respectively. Using the same computations as those in (1) and Case 1, we get
\begin{equation}\label{converge6}
u_{\infty}(x,t;\lambda)=u_{\infty}(x,t+T;\lambda)\quad {\rm for\ any}\ (x,t)\in\Bbb{R}^2
\end{equation}
and
\begin{equation}\label{converge12}
  u_{\infty}(x,t+T;\lambda)\geq u_{\infty}(x+M,t;\lambda)\quad {\rm for\ any}\ (x,t)\in\Bbb{R}^2\ {\rm and}\ M<0.
\end{equation}
By \eqref{converge6}, \eqref{converge12} and the fact $\partial_xu_{\infty}(x,t;\lambda)\leq 0$ for any $(x,t)\in\Bbb{R}$, we have  that $u_{\infty}(x,t;\lambda):=u_{\infty}(t;\lambda)$ is independent of $x\in\Bbb{R}$ and $u_{\infty}(x,t;\lambda):=u_{\infty}(t;\lambda)$ is a $T$-periodic solution.
This completes the proof.
\end{proof}

\begin{remark}\label{re2convergence}
{\rm By Lemma \ref{le2convergence} and its proof, we have that if $u_\infty(x,t;\lambda)$ is a pulsating traveling front of \eqref{s11}, then there exists a $L>0$ such that $\ell_j\to L$ as $j\to\infty$ and
\[
 u_\infty(x-L,t;\lambda)=u_\infty(x,t+T;\lambda)\quad {\rm for\ any}\ (x,t)\in\Bbb{R}^2.
\]
$c=\frac{L}{T}$ is the speed of the pulsating traveling front. By the hypothesis  (H1), a $T$-periodic solution $\omega(\alpha,t)$ is isolated from below. If $\lambda \in (0,\alpha)$ is close enough to $\alpha$, then $u_\infty(x,t;\lambda)$ can not be a $T$-periodic solution. In fact, $u_\infty(x,t;\lambda)$ must be a pulsating traveling front of \eqref{s11} connecting the $T$-periodic solution $\omega_0(t):=\omega(\alpha,t)$ to another $T$-periodic solution $\omega_1(t)<\omega_0(t)$ with a positive speed $c_1$.  }
\end{remark}

\section{Convergence of solution to a propagating terrace}
 In this  section, we prove the convergence of the solutions to a propagating terrace. In Subsection 4.1,  we construct a minimal propagating terrace  in the sense of Definition \ref{def1terrace}. In Subsection 4.2, we investigate the convergence of the solutions to the propagating terrace.

\subsection{Existence of a minimal propagating terrace}
In this subsection we prove that there exists  a minimal propagating terrace of \eqref{s11} which is unique, namely, we prove Theorem \ref{th1terrace}. The method is to use  iterative arguments. As mentioned in Remark \ref{re2convergence}, here we first choose a $\lambda_1\in (0,\alpha)$ close enough to $\alpha$ to get a pulsating traveling front $U_1(x,t):=u_\infty(x,t;\lambda_1)$ of \eqref{s11} connecting the $T$-periodic solution $\omega_0(t):=\omega(\alpha,t)$ to another $T$-periodic solution $\omega_1(t)<\omega_0(t)$ with a positive speed $c_1$, which is  the first step of such a terrace. Then we have the following lemma.

\begin{lemma}\label{le1terrace}
For $0<\lambda_{k}<\alpha$, if  $U_{k}(x,t):=u_{\infty}(x,t;\lambda_{k})$ is a pulsating traveling front connecting $T$-periodic solutions $\omega_{k-1}(t)$ and $\omega_{k}(t)<\omega_{k-1}(t)$, then one has

(1) $\omega_{k}(t)$ is isolated and stable from below with respect to \eqref{s13};

(2) there exists $\lambda_{k+1}\in(0,\omega_{k}(0))$ such that $U_{k+1}:=u_\infty(x,t;\lambda_{k+1})$ is a pulsating traveling front  connecting $\omega_{k}(t)$ to some $T$-periodic solution $\omega_{k+1}(t)<\omega_{k}(t)$.
\end{lemma}

The proof of Lemma \ref{le1terrace} relies strongly on the following two lemmas.

\begin{lemma}\label{le2terrace} Assume that $U_{k}(x,t):=u_{\infty}(x,t;\lambda_{k})$ with $0<\lambda_{k}<\alpha$ is a pulsating traveling front of \eqref{s11} connecting $T$-periodic solutions $\omega_{k-1}(t)$ to $\omega_{k}(t)<\omega_{k-1}(t)$, then
$\omega_{k}(t)$ is steeper than any other entire solution of \eqref{s11}. Moreover,
$\omega_{k}(t)\equiv u_{\infty}(x,t;\omega_{k}(0))$.
\end{lemma}
\begin{proof}
Assume that $v$ is an entire solution of \eqref{s11} and $0<v(x,t)<\omega(\alpha,t)$.  By Lemma \ref{lem-limitset}, $U_{k}(x,t):=u_{\infty}(x,t;\lambda_{k})$ is steeper than $v(x,t)$, so we have
\begin{equation}\label{s41}
Z[u_{\infty}(\cdot,t'-jT;\lambda_{k})-v(\cdot,t')]\leq1\ \ {\rm and}\ \
SGN[u_{\infty}(\cdot,t'-jT;\lambda_{k})-v(\cdot,t')]\lhd[+\,-]
\end{equation}
 for any $t'\in\mathbb{R}$ and $j\in\mathbb{Z}$. It follows from Remark \ref{re2convergence} that there exists $L_k>0$ such that
 \[
 u_{\infty}(x-L_k,t;\lambda_{k})=u_{\infty}(x,t+T;\lambda_{k})\quad {\rm for\ any}\ (x,t)\in\Bbb{R}^2,
 \]
 hence, we have
 \[
 u_{\infty}(x-jL_k,t;\lambda_{k})=u_{\infty}(x,t+jT;\lambda_{k})\quad {\rm for\ any}\ (x,t)\in\Bbb{R}^2,\ j\in\Bbb{Z}.
 \]
  Letting  $j\to\infty$ in \eqref{s41} yields
\[
Z[\omega_{k}(t')-v(\cdot,t')]\leq\liminf_{j\to\infty}Z[u_{\infty}(\cdot,t'-jT;\lambda_{k})-v(\cdot,t')]\leq1,
\]
\[
SGN[\omega_{k}(t')-v(\cdot,t')]\lhd\liminf_{j\to\infty}Z[u_{\infty}(\cdot,t'-jT;\lambda_{k})-v(\cdot,t')]\lhd[+\,-].
\]
Due to the periodicity of $\omega_{k}$, we have $\omega_{k}(t')=\omega_{k}(t'+iT)$ for any $i\in\mathbb{Z}$. It then follows from Proposition \ref{prosteeper} that $\omega_{k}$ is steeper than $v$. By the arbitrariness of $v$, we have that $\omega_{k}(t)$ is steeper than $u_{\infty}(x,t;\omega_{k}(0))$. From Lemma \ref{lem-limitset}, we also have that $u_{\infty}(x,t;\omega_{k}(0))$ is steeper than $\omega_{k}(t)$. Since $u_{\infty}(0,0;\omega_{k}(0))
=\omega_{k}(0)$, we finally obtain $\omega_{k}(t)\equiv u_{\infty}(x,t;\omega_{k}(0))$. This completes the proof.
\end{proof}

\begin{lemma}\label{le3terrace}
Let $v(x,t)$ be a supersolution of \eqref{s11} on the domain $D=\{(x,t)\,|\,x\geq x(t)\}$ and $0<v(x,t)<\omega(\alpha,t)$. If
$v(x(t),t)=\omega_{k}(t)$ for all $t\in\mathbb{R}$, where $x(t)$ moves with the average speed $c\,(0<c<c_{\ast})$ such that $x(0)>a$ and
\[x(t+jT)\leq x(t+(j+1)T),\,\forall t\in\mathbb{R},j\in\mathbb{N},\]
then there exist $t_{j}\in[0,T]$ and $k_{j}\to\infty\,(j\to\infty)$ such that for any $x\geq0$, one has
\[
\liminf_{j\to\infty}\left[\,v(x(t_{j}+k_{j}T)+x,t_{j}+k_{j}T)-\omega_{k}(t_{j}+k_{j}T)\,\right]\geq0.
\]
\end{lemma}
\begin{proof}
Combining the properties of $x(t)$ with Lemma \ref{le1speed} (2), we know that
\[
\tilde{u}(x(t),t-jT;a)\to\omega(\alpha,t)\,(t\to\infty)\quad {\rm for\ any} \ j\in\Bbb{N}.
\]
It is easy to come that $0=\tilde{u}(x,0;a)<v(x,0)\,(x\geq x(jT) \geq x(0)\geq a)$ and $v(x,t)<\omega(\alpha,t)$. Then there exist $t_{j}\in[0,T)$ and $k_{j}\in\Bbb{N}$ with $k_j\geq j$ such that
\begin{equation}\label{terrace1}
 \tilde{u}(x(t_{j}+k_{j}T),t_{j}+k_{j}T-jT;a)=v(x(t_{j}+k_{j}T),t_{j}+k_{j}T)=\omega_{k}(t_{j}+k_{j}T).
\end{equation}
Since $v$ is a super-solution of \eqref{s11} on $D$, then by the comparison principle \cite{Friedman},  for any $j\in\Bbb{N}$ one has
 \[\tilde{u}(x,t_{j}+k_{j}T-jT;a)\leq v(x,t_{j}+k_{j}T) \,\,\,\,\text{for}\,\,x\geq x(t_{j}+k_{j}T).\]
Due to \eqref{terrace1} and Lemma \ref{le1speed} (2), we conclude that $k_j-j\to +\infty$ as $j\to +\infty$. Assume that $t_{j}\to t'\in[\,0,T\,]\,(j\to\infty)$, then
we have
\[
\tilde{u}(x+x(t_{j}+k_{j}T),t+t_{j}+k_{j}T-jT;a)\to v_{\infty}(x,t+t')\quad {\rm as}\ j\to\infty.
\]
Using \eqref{terrace1}, we easily  get that
\begin{equation}\label{terrace3}
v_{\infty}(0,t')=\lim_{j\to\infty}\tilde{u}(x(t_{j}+k_{j}T),t_{j}+k_{j}T-jT;a)=\lim_{j\to\infty}\omega_{k}(t_{j}+k_{j}T)=\omega_{k}(t').
\end{equation}
Thus, we have $v_{\infty}$ is a limit orbit of $\tilde{u}$. Therefore, $v_{\infty}$ is steeper than $\omega_{k}$. On the other hand, it follows from Lemma \ref{le2terrace} that $\omega_{k}$ is steeper than $v_{\infty}$. By \eqref{terrace3}, we have $\omega_{k}\equiv v_{\infty}$. Hence, for any $x\geq0$, we obtain that
\[
\liminf_{j\to\infty}\left[\,v(x(t_{j}+k_{j}T)+x,t_{j}+k_{j}T)-\omega_{k}(t_{j}+k_{j}T)\,\right]\geq0.
\]
This completes the proof.
\end{proof}\\

\noindent \textbf{Proof of Lemma \ref{le1terrace}:} We completes the proof by three steps:

{\it Step $1$. $\omega_{k}(t)$ is isolated from below.}

If not, we assume that there exists the sequence $\{h_{j}\}_{j\in\mathbb{N}}$ of $T$-periodic solutions such that
\begin{equation}\label{s42}
h_{j}(t)<\omega_{k}(t)\quad {\rm and}\
\,\,\,h_{j}(t)\to\omega_{k}(t)\,(j\to\infty) \quad {\rm uniformly\ in}\  \Bbb{R}.
\,\end{equation}

 Let $\mu_k:=-\frac{1}{T}\int_{0}^{T}f_{u}(t,\omega_{k}(t))dt$ and
 \begin{equation}\label{s43}
 \varphi_k(t)=\exp\left(\int_{0}^{t}f_{u}(s,\omega_{k}(s))ds-\frac{t}{T}\int_{0}^{T}f_{u}(t,\omega_{k}(t))dt \right).
 \end{equation}
 Clearly, $\mu_k$ and $\varphi_k(t)$ are the eigenvalue and the eigenfunction of the  following periodic eigenvalues problem:
 \begin{equation}\label{eigenvalue1}
 \begin{cases}
 \varphi'(t)-f_{u}(t,\omega_{k}(t))\varphi(t)=\mu\varphi(t),\quad \quad \,\,\forall t\in\mathbb{R},\\
 \varphi(t)>0\,\,\text{and}\,\,\varphi(t)=\varphi(t+T),\quad\quad\quad\,\forall t\in\mathbb{R}.
 \end{cases}
 \end{equation}
Due to \eqref{s42}, there must be $\mu_k\geq 0$. Let $v(x,t)=\min\{\omega_{k}(t),\,\varphi_k(t) e^{-\lambda(x-ct)}+h_{j}(t)\}$, where $0<c<c_{\ast}$ and $\lambda>0$. Clearly, there exists a function $x(t)$, which moves with the average speed $c\,(0<c<c_{\ast})$ and satisfies $x(t+iT)<x(t+(i+1)T)$ for any $t\in\Bbb{R}$ and $i\in\Bbb{Z}$,  such that
\[
v(x(t),t)=\omega_{k}(t) \quad {\rm  for\ any}\ t\in\mathbb{R}\quad {\rm and}\quad
v(x,t)<\omega_{k}(t) \quad {\rm  for\ any}\  x>x(t),\ t\in\mathbb{R}.
\]
Define the domain $D=\{(x,t)\,|\,x\geq x(t)\}$. It is easy to verify that $v(x,t)$ is a super-solution of \eqref{s11} on $D$ for   $\lambda>0$ small enough and $j\in\Bbb{N}$ large enough. Namely, for   $\lambda>0$ small enough and $j\in\Bbb{N}$ large enough,  one has
\begin{equation*}
\begin{aligned}
&v_{t}-v_{xx}-f(t,v)\\
=&e^{-\lambda(x-ct)}\left(c\lambda\varphi_k(t)+\varphi_k^\prime(t)\right)+h^\prime_{j}(t)-\lambda^{2}e^{-\lambda(x-ct)}\varphi_k(t)-f\left(t,\varphi_k(t) e^{-\lambda(x-ct)}+h_{j}(t)\right)\\
=&e^{-\lambda(x-ct)}\left(c\lambda\varphi_k(t)+\varphi_k^\prime(t)-\lambda^{2}\varphi_k(t)\right)-f_{u}(t,h_{j}(t))e^{-\lambda(x-ct)}\varphi_k(t)+o\left(e^{-\lambda(x-ct)}\varphi_k(t)\right)\\
=& e^{-\lambda(x-ct)} \left(c\lambda\varphi_k(t)+\mu_k\varphi_k(t)+f_{u}(t,\omega_{k}(t))\varphi_k(t)-\lambda^{2}\varphi_k(t)\right)-f_{u}(t,h_{j}(t))e^{-\lambda(x-ct)}\varphi_k(t)\\
& + o\left(e^{-\lambda(x-ct)}\varphi_k(t)\right) \\
=&e^{-\lambda(x-ct)}\varphi_k(t)(c\lambda-\lambda^{2}+\mu_k)+\left(f_{u}(t,\omega_{k}(t))-f_{u}(t,h_{j}(t))\right)e^{-\lambda(x-ct)}\varphi_k(t)\\
&+o\left(e^{-\lambda(x-ct)}\varphi_k(t)\right)\\
\geq& 0
\end{aligned}
\end{equation*}
for any $(x,t)\in D$. Now by Lemma \ref{le3terrace}, there exist $k_{i}\to\infty\,(i\to\infty)$ and $t_{i}\in[\,0,T\,]$ such that
\[
\liminf_{i\to\infty}\left[v(x(t_{i}+k_{i}T)+x,t_{i}+k_{i}T)-\omega_{k}(t_{i}+k_{i}T)\right]\geq0,\quad\forall x\geq0.
\]
However, by the definition of $v$, there exists $M>0$ such that $v(x(t)+M,t)<\omega_{k}(t)$ for any $t\in\Bbb{R}$. This is a contradiction. Therefore, $\omega_{k}(t)$ is isolated from below.\\

{\it Step $2$.  $\omega_{k}(t)$ is stable from below.}

Suppose on the contrary that $\omega_{k}(t)$ is unstable from below. Then $\mu_k\geq 0$. Since $\omega_{k}(t)$ is isolated from below with respect to \eqref{s13},  it follows from the assumption (H2) that there exists a sequence $\{\phi_{k,j}(t)\}_{j\in\Bbb{N}}$ of supersolutions of \eqref{s13} such that $\omega_{k+1}(t)<\phi_{k,j}(t)<\omega_k(t)$, $\phi_{k,j}(t)$ is $T$-periodic, and $\phi_{k,j}(t)$ converges to $\omega_k(t)$ uniformly in $t\in\Bbb{R}$ as $j\to\infty$.

 Let $v(x,t)=\min\{\omega_{k}(t),\,\varphi_k(t) e^{-\lambda(x-ct)}+\phi_{k,j}(t)\}$, where $0<c<c_{\ast}$ and $\lambda>0$. As before, there exists a function $x(t)$ such that
\[
v (x (t),t)=\omega_{k}(t) \quad {\rm  for\ any}\ t\in\mathbb{R}\quad {\rm and}\quad
v (x,t)<\omega_{k}(t) \quad {\rm  for\ any}\  x>x (t),\ t\in\mathbb{R}.
\]
In particular, $x (t)$ moves with the average speed $c\,(0<c<c_{\ast})$ and satisfies $x (t+iT)<x (t+(i+1)T)$ for any $t\in\Bbb{R}$ and $i\in\Bbb{Z}$.
Define the domain $D =\{(x,t)\,|\,x\geq x (t)\}$. For   $\lambda>0$ small enough and $j\in\Bbb{N}$ large enough,  one has
\begin{equation*}
\begin{aligned}
& v_{t}- v_{xx}-f(t,v )\\
=&e^{-\lambda(x-ct)}\left(c\lambda\varphi_k(t)+\varphi_k^\prime(t)\right)+\phi_{k,j}^\prime (t)-\lambda^{2}e^{-\lambda(x-ct)}\varphi_k(t)-f\left(t,\varphi_k(t) e^{-\lambda(x-ct)}+\phi_{k,j}(t)\right)\\
\geq&e^{-\lambda(x-ct)}\left(c\lambda\varphi_k(t)+\varphi_k^\prime(t)-\lambda^{2}\varphi_k(t)\right)+f(t,\phi_{k,j}(t)) -f\left(t,\varphi_k(t) e^{-\lambda(x-ct)}+\phi_{k,j}(t)\right)\\
=& e^{-\lambda(x-ct)} \left(c\lambda\varphi_k(t)+\mu_k\varphi_k(t)+f_{u}(t,\omega_{k}(t))\varphi_k(t)-\lambda^{2}\varphi_k(t)\right)-f_{u}(t,\phi_{k,j}(t))e^{-\lambda(x-ct)}\varphi_k(t)\\
& + o\left(e^{-\lambda(x-ct)}\varphi_k(t)\right) \\
=&e^{-\lambda(x-ct)}\varphi_k(t)(c\lambda-\lambda^{2}+\mu_k)+\left(f_{u}(t,\omega_{k}(t))-f_{u}(t,\phi_{k,j}(t))\right)e^{-\lambda(x-ct)}\varphi_k(t)\\
&+o\left(e^{-\lambda(x-ct)}\varphi_k(t)\right)\\
\geq& 0
\end{aligned}
\end{equation*}
for any $(x,t)\in D $, which implies that $v (x,t)$ is a supersolution of \eqref{s11} on $D $ for   $\lambda>0$ small enough and $j\in\Bbb{N}$ large enough. It follows from Lemma \ref{le3terrace} that there exist $k_{i}\to\infty\,(i\to\infty)$ and $t_{i}\in[\,0,T\,]$ such that
\[
\liminf_{i\to\infty}\left[v (x (t_{i}+k_{i}T)+x,t_{i}+k_{i}T)-\omega_{k}(t_{i}+k_{i}T)\right]\geq0,\quad\forall x\geq0.
\]
But by the definition of $v$, there exists $M>0$ satisfying $v(x(t)+M,t)<\omega_{k}(t)$ for any $t\in\Bbb{R}$. This is a contradiction. Therefore, $\omega_{k}(t)$ is stable from below. \\

{\it Step $3$.  $U_{k+1}(x,t)$ is a pulsating traveling wave connecting $\omega_{k+1}(t)$ to $\omega_{k}(t)$.}

Let $\lambda_{k+1}\in(0,\omega_{k}(0))$ close enough to $\omega_{k}(0)$. By Lemmas \ref{le1convergence} and \ref{le2convergence}, $U_{k+1}(x,t):=u_{\infty}(x,t;\lambda_{k+1})$ is  a pulsating traveling wave with a positive speed $c_{k+1}$ connecting $\omega_{k}(t)$ to another $T$-periodic solution $\omega_{k+1}(t)$, where $\omega_{k+1}(t)<\omega_{k}(t)$. In particular, there is
\[
\lim_{x\to-\infty}U_{k+1}(x,t)=\omega_k(t)\quad {\rm and}\quad \lim_{x\to+\infty}U_{k+1}(x,t)=\omega_{k+1}(t).
\]
This completes the proof of Lemma \ref{le1terrace}. $\square$\\

\noindent \textbf{Proof of Theorem \ref{th1terrace}:} According to the argument above, we know that there exist a sequence $\{\omega_k(t)\}_{k=0}^\infty$ of $T$-periodic solutions of \eqref{s11} and a sequence $\{U_k(x,t)\}_{k=1}^\infty$ of pulsating traveling fronts of \eqref{s11}, which satisfy

\noindent (a) $0\leq \cdots<\omega_{k+1}(t)<\omega_k(t)<\cdots<\omega_0(t):=\omega(\alpha,t)$ for any $t\in\Bbb{R}$.

\noindent (b)  $\omega_k(t)$ is isolated and stable form below.

\noindent (c) $U_k(x,t)$ is the pulsating traveling front of \eqref{s11} with  speed $c_k>0$ connecting $\omega_{k-1}(t)$ and $\omega_k(t)$. Moreover, $\partial_xU_k(x,t)<0$.

\noindent (d) $U_k(x,t)=u_\infty(x,t;\lambda_k)$, $\alpha>\lambda_1>\omega_1(0)>\lambda_2>\omega_2(0)>\cdots>\lambda_N>0$, $c_k=\frac{L_k}{T}$.

\noindent (e) $U_k(x,t)=u_\infty(x,t;\lambda_k)$ and $\omega_k(t)$ are steeper than any other entire solutions.

To complete the proof of Theorem \ref{th1terrace}, we divide the remainder of the proof into three parts:

 (i) Sequence $\{\omega_k(t)\}_{k=0}^\infty$ and  $\{U_k(x,t)\}_{k=1}^\infty$ are finite, namely, there exists $N\in\mathbb{N}$ such that $\omega_{N}(t)\equiv0$.

On the contrary, we assume that   the sequence $\{\omega_k(t)\}_{k=0}^\infty$ is infinite. Due to the above (a), there exists a $T$-periodic solution $\omega_{\infty}(t)$ of \eqref{s11}, which satisfies
\[
\lim_{k\to\infty}\omega_{k}(t)=\omega_{\infty}(t)\quad {\rm and}\quad 0\leq \omega_{\infty}(t)<\omega_{k}(t)\quad {\rm for\ any}\ t\in\mathbb{R}.
\]
Let $\varphi(t)$ and $\mu$ be defined as before.
Let $v(x,t)=\min\{\omega_{k}(t),\,\varphi(t) e^{-\lambda(x-ct)}+\omega_{\infty}(t)\}$, where $0<c<c_{\ast},\lambda>0$. It is easy to see that there exists a function $x(t)$, which moves with the average speed $c $ and satisfies $x(t+iT)<x(t+(i+1)T)$ for any $t\in\Bbb{R}$ and $i\in\Bbb{Z}$,  such that
\[
v(x(t),t)=\omega_{k}(t) \quad {\rm  for\ any}\ t\in\mathbb{R}\quad {\rm and}\quad
v(x,t)<\omega_{k}(t) \quad {\rm  for\ any}\  x>x(t),\ t\in\mathbb{R}.
\]
Let $D=\{(x,t)\,|\,x\geq x(t)\}$. As before, for any $(x,t)\in D$, one has
\begin{equation*}
\begin{aligned}
&v_{t}-v_{xx}-f(t,v)\\
=&e^{-\lambda(x-ct)}\varphi(t)(c\lambda-\lambda^{2}+\mu)+(f_{u}(t,\omega_{k}(t))-f_{u}(t,\omega_\infty(t))e^{-\lambda(x-ct)}\varphi(t) \\
&+o\left(e^{-\lambda(x-ct)}\varphi(t)\right)\\
\geq& 0,
\end{aligned}
\end{equation*}
if we take  $\lambda$ small enough and $k\in\Bbb{N}$ large enough. Using Lemma \ref{le3terrace}, we also get a contradiction.

(ii) Sequence $c_{k}$ of the speeds  is nondecreasing, namely, $0<c_1\leq c_2\leq \cdots\leq c_N$.

By virtue of Remark \ref{re2convergence}, we have
\[
c_{k}=\frac{L_k}{T}=\lim_{j\to\infty}\frac{\ell(j,\lambda_{k})}{jT}.
\]
It follows from Lemma \ref{lem3.1} and $\ell(j,\lambda)\to \infty$ ($j\to\infty$) that $\partial_x\tilde{u}(x,t;a)<0$, which implies that $\ell(j,\lambda)$ is decreasing on $\lambda\in (0,\alpha)$.
Thus, due to $\lambda_1>\lambda_2>\cdots>\lambda_N$, we obtain $0<c_1\leq c_2\leq \cdots\leq c_N$.

$\bf{(iii)}$  Propagating terrace $Q=\{ \omega_{k}(t)_{0\leq k\leq N},U_{k}(x,t)_{1\leq k\leq N}\}$ is minimal and unique.

 We  first show  that  $Q$ is  minimal in the sense of Definition  \ref{def1terrace}. Suppose on the contrary that $Q$ is not minimal, namely, there exists a propagating terrace
 \[
 P=\{\psi_{j}(t)_{0\leq j\leq N_{0}},\,V_{j}(x,t)_{1\leq j\leq N_{0}}\}
 \]
 such
that there exist some pulsating traveling front $V_{j_0}\in P$ and some $k_0$ such that $V_{j_0}(x,t)$ connects two periodic solutions $\psi_{j_0+1}(t) $ and $\psi_{j_0}(t) $ and $\psi_{j_0+1}(t)<\omega_{k_0}(t)<\psi_{j_0}(t)$ for any $t\in\Bbb{R}$. Consequently, there exists some $x_0\in\Bbb{R}$ such that $V_{j_0}(x_0,0)=\omega_{k_0}(0)$. Let $\tilde{V}_{j_0}(x,t):=V_{j_0}(x,t)$. Clearly, $\tilde{V}_{j_0}(x,t)$ is  also the pulsating traveling front of \eqref{s11} connecting two periodic solutions $\psi_{j_0+1}(t)$ and $\psi_{j_0}(t)$.  Since $V_{j_0}(x,t)$ is steeper than any other entire solution of \eqref{s11}, $\tilde{V}_{j_0}(x,t)$ is also steeper than any other entire solutions.  At the same time,  $\omega_{k_0}(t)$ is steeper than any other entire solution. By Definition \ref{steeper}, we have $\tilde{V}_{j_0}(x,t)=\omega_{k_0}(t)$, which is a contradiction.

Now we show that $Q$ is unique in the sense that any minimal propagating terrace has the same $\omega_k(t)_{0\leq k\leq N}$.  Let $\widetilde{P}=\{\psi_{k}(t)_{0\leq k\leq N^\prime},\,V_{k}(x,t)_{1\leq k\leq N^\prime}\}$  be a minimal propagating terrace in the sense of Definition \ref{def1terrace}. Since both $\widetilde{P}$ and $Q$ are minimal in the sense of Definition \ref{def1terrace}, it follows that  $\{\omega_{k}(t)\,|\,0\leq k\leq N\}$ is equal to $\{\psi_{k}(t)\,|\,0\leq k\leq N^\prime\}$, which implies the uniqueness of $Q$. Thus, we have $\widetilde{P}=\{\omega_{k}(t)_{0\leq k\leq N},\,V_{k}(x,t)_{1\leq k\leq N}\}$.  For each $k$, both $V_{k}(x,t)$ and $U_{k}(x,t)$ are pulsating traveling fronts of \eqref{s11} connecting two periodic solutions $\omega_{k+1}(t)$ and $\omega_{k}(t)$. Since they  are steeper  than each other, then by an argument as above we have that $V_{k}(\cdot+x_0,\cdot)\equiv U_{k}(\cdot,\cdot)$ for some $x_0\in\Bbb{R}$. This further implies that the pulsating traveling front $U_{k}(x,t)$ is unique up to  spatially translations. This completes the proof.
 $\square$

\subsection{Convergence to the minimal propagating terrace}

 In this subsection, let us show the convergence result, namely, we prove Theorem \ref{th2terrace}.

\noindent \textbf{Proof of Theorem \ref{th2terrace}:}
(1) We first prove that the solution $\tilde{u}(x,t;a)$ of \eqref{s11} converges to pulsating traveling fronts $\{U_{k}(x,t)\}_{1\leq k \leq N}$ along the moving frames with speed $c_{k}$ and some sublinear drifts, namely, we prove the statement (1) of Theorem \ref{th2terrace}.

Fix $1\leq k\leq N$. For  sufficiently large $t$, we define $j(t)\in\mathbb{N}$ such that $j(t)T\leq t<(j(t)+1)T$, and a piecewized affine function  $g_k(t)$ with $g_{k}(t)=j(t)T-\frac{1}{c_{k}}\ell(j(t),\lambda_k)$. Recall the proof of Theorem \ref{th1terrace}, the sequence $\left\{\frac{1}{j}\sum^{i=j}_{i=1}\ell_{i}\right\}_{j=1}^{\infty}$, that is $\left\{\frac{1}{j} \ell(j,\lambda_k)\right\}_{j=1}^{\infty}$ converges to $c_{k}T$ as $j\to\infty$. Therefore, $g_{k}(t)=o(j(t))=o(t)$ as $t\rightarrow \infty$. In view of Lemmas \ref{le1convergence} and \ref{le1terrace}, we have
\[
\tilde{u}(x,t+jT;-\ell(j,\lambda_{k}))\to U_{k}(x,t)\,\,(j\to\infty)\,\,\rm{locally \,\,uniformly \,\,on}\,\, \mathbb{R}^{2}.
\]
Since $\tilde{u}(x-a,t+jT;-\ell(j,\lambda_{k}))=\tilde{u}(x+\ell(j,\lambda_{k}),t+jT;a)$, we then have
\[
\tilde{u}(x+\ell(j,\lambda_{k}),t+jT;a)\to U_{k}(x-a,t)\,\,(j\to\infty)\,\,\rm{locally \,\,uniformly \,\,on}\,\, \mathbb{R}^{2}.
\]
Consequently, we have
\[
\tilde{u}(x+\ell(j(t),\lambda_{k}),t;a)\to U_{k}(x-a,t-j(t)T)\quad {\rm as}\ t\to\infty \,\,{\rm locally \,\,uniformly \, on}\, x\in\mathbb{R}.
\]
Since $ t-g_{k}(t)-\frac{1}{c_{k}} \ell(j(t),\lambda_{k}) = t-j(t)T $, we then have
\begin{eqnarray*}
&&\tilde{u}(x+c_k(t-g_{k}(t)),t;a)\\
&\to& U_{k}(x-a+c_k(t-g_{k}(t)) -\ell(j(t),\lambda_{k}),t-j(t)T)\\
&=&U_{k}(x-a+c_k(t-g_{k}(t)) -\ell(j(t),\lambda_{k})+c_kj(t)T,t)\\
&=&U_{k}(x+c_kt-a,t)
\end{eqnarray*}
as $t\to\infty$  locally  uniformly  on $x\in\mathbb{R}$.

(2) We prove the statement (2) of Theorem \ref{th2terrace}.

Consider $x+c_{1}(t-g_{1}(t))\to-\infty$. Since $\partial_xU_{1}(x,t)<0$ and $U_{1}(x-c_1T,t)=U_{1}(x,t+T)$ for any $(x,t)\in\Bbb{R}^2$ and
\[
\lim_{x\to -\infty}U_{1}(x,t)=\omega(\alpha,t) \quad {\rm uniformly\ in}\ t\in [0,+\infty),
\]
for any $n\in\Bbb{N}$, then for any $\varepsilon>0$, there exists $C_{0}>0$ such that for any $x\leq-C_{0}$ and $t\in\Bbb{R}$,
\begin{equation}\label{tc2}
\omega(\alpha,t)-\frac{\varepsilon}{2}\leq
U_{1}(x+c_{1}t-a,t)\leq\omega(\alpha,t).
\end{equation}
According to the statement (1), there exists  and $n_{0}\in\mathbb{N}$ such that for any $x\in[-C_{0},C_{0}]$ and $t\geq n_{0}T$,
\begin{equation}\label{tc6}
|\tilde{u}(x+c_{1}(t-g_{1}(t)),t;a)-U_{1}(x+c_{1}t-a,t)|\leq\frac{\varepsilon}{2}.
\end{equation}
Since $\tilde{u}(x,t;a)$ is nonincreasing with respect to $x$,  then by  \eqref{tc2}, \eqref{tc6} and the periodicity of $U_{1}$ and $\omega(\alpha,t)=\omega_0(t)$, we have
\[
\omega(\alpha,t)-\varepsilon\leq \tilde{u}\left(x+c_{1}(t-g_{1}(t)),t;a\right)\leq\omega(\alpha,t)
\]
for any $x\leq-C_{0}$ and $t\geq S_{0}T$, namely, we have
\[
\|\tilde{u}(\cdot,t;a)-\omega(\alpha,t)\|_{L^{\infty}(I_{0,c}(t))}\leq\varepsilon\quad  {\rm for\ any}\  t\geq n_{0}T\,\, {\rm with}\ \,\,C:=C_0.
\]

Similarly, we can show that, for any $\varepsilon>0$, there exist $C_N>0$ and $n_{N}\in\mathbb{N}$ such that
\[
\|\tilde{u}\left(\cdot,t;a\right)\|_{L^{\infty}(I_{N,c}(t))}\leq\varepsilon \quad {\rm for\ any}\  t\geq n_{N}T\,\, {\rm with}\ \,\,C:=C_N.
\]

Finally, fix $1\leq k< N$. Similar to $\eqref{tc2}$, for any $\varepsilon>0$, there exists $C_k>0$ such that
\[
\omega_k(t)\leq
U_{k}(x+c_{k}t-a,t)\leq\omega_k(t)+\frac{\varepsilon}{2}\quad {\rm for\ any }\ x>C_k,\ t\in\Bbb{R}
\]
and
\[
\omega_k(t)-\frac{\varepsilon}{2}\leq
U_{k+1}(x+c_{k+1}t-a,t)\leq\omega_k(t) \quad {\rm for\ any }\  x<-C_k,\ t\in\Bbb{R}.
\]
Similar to \eqref{tc6}, there exists   $n_{k}\in\mathbb{N}$ such that
\begin{equation*}
|\tilde{u}(x+c_{k}(t-g_{k}(t)),t;a)-U_{k}(x+c_{k}t-a,t)|\leq\frac{\varepsilon}{2}
\end{equation*}
and
\begin{equation*}
|\tilde{u}(x+c_{k+1}(t-g_{k+1}(t)),t;a)-U_{k+1}(x+c_{k+1}t-a,t)|\leq\frac{\varepsilon}{2}
\end{equation*}
 for any $x\in[-C_{k},C_{k}]$ and $t\geq n_{k}T$,
Thus, we have
\[
\|\tilde{u}(\cdot,t;a)-\omega_{k}(t)\|_{L^{\infty}(I_{k,c}(t))}\leq\varepsilon \quad
 {\rm for\ any}\  t\geq n_{k}T \,\,{\rm with}\ \,\, C:=C_k.
\]
This completes the proof.
$\square$

\section*{Acknowledgments}
 This work was supported by  NNSF of China (11371179, 11731005).


\begin{thebibliography}{99}\setlength{\itemsep}{-1mm}\linespread{1}\selectfont

\bibitem{ABC1999TAMS} N. Alikakos, P. Bates, X. Chen, Periodic traveling waves and locating oscillating patterns in multidimensional domains, \textit{Trans. Amer. Math. Soc.}  $\mathbf{351}$  (1999)  2777-2805.


\bibitem{SA} S. Angenent, The zero set of a solution of a parabolic equation, \textit{ J. Reine Angew. Math.} $\mathbf{390}$ (1988)  79-96.

 \bibitem{AW1978} D.G. Aronson, H.F. Weinberger, Multidimensional nonlinear diffusion arising in population genetics, \textit{Adv. in Math.}  $\mathbf{30}$  (1978) 33-76.





\bibitem{BW2013JDE} X. Bao, Z.-C. Wang, Existence and stability of time periodic traveling waves for a periodic bistable Lotka-Volterra competition system, \textit{J. Differential Equations}  $\mathbf{255}$  (2013)  2402-2435.


\bibitem{B2019RWA} W.-J. Bo, G. Lin, Y.-W. Qi, Propagation dynamics of a time periodic diffusion equation with degenerate nonlinearity, \textit{Nonlinear Anal. Real World Appl.}  $\mathbf{45}$  (2019) 376-400.

\bibitem{B1983} M. Bramson, Convergence of solutions of the Kolmogorov equation to travelling waves,
\textit{Mem. Amer. Math. Soc.} $\mathbf{44}$  (1983) iv+190.

\bibitem{W2016RWA}  Z.-H. Bu, Z.-C. Wang, N.-W. Liu, Asymptotic behavior of pulsating fronts and entire solutions of reaction-advection-diffusion equations in periodic media, \textit{Nonlinear Anal. Real World Appl.}  $\mathbf{28}$  (2016) 48-71.

\bibitem{DingM2018} W. Ding, H. Matano, Dynamics of time-periodic reaction-diffusion equations with compact initial support on $\mathbb{R}$, arXiv:1807.04146  [math.AP].


\bibitem{DM2010JEMS} Y. Du, H. Matano, Convergence and sharp thresholds for propagation in nonlinear diffusion problems, \textit{J. Eur. Math. Soc.}  $\mathbf{12}$  (2010)  279-312.

\bibitem{DM2018Pre} Y. Du, H. Matano, Radial terrace solutions and propagation profile of multistable reaction-diffusion equations over $\mathbb{R^{N}}$,  arXiv:1711.00952 [math.AP].

\bibitem{DGM} A. Ducrot, T. Giletti, H. Matano, Existence and convergence to a propagating terrace in one-dimensional reaction-diffusion equations, \textit{Trans. Amer. Math. Soc.} $\mathbf{366}$  (2014)  5541-5566.

\bibitem{DP} Y. Du, P. Polacik, Locally uniform convergence to an equilibrium for nonlinear parabolic equations on $\mathbb{R}^{N}$, \textit{Indiana Univ. Math. J.} $\mathbf{64}$ (2015) 787-824.



 \bibitem{FZ} J. Fang, X.-Q. Zhao, Bistable traveling waves for monotone semiflows with applications, \textit{J. Eur. Math. Soc.} $\mathbf{17}$  (2015)   2243-2288.

\bibitem{FM1977ARMA} P.C. Fife, J.B. Mcleod, The approach of solutionts of nonlinear diffusion equations to travelling front solutions, \textit{Arch. Ration. Mech. Anal.} $\mathbf{65}$  (1977) 335-361.

\bibitem{Friedman} A. Friedman, Partial differential equations of parabolic type. Prentice-Hall, Inc., Englewood Cliffs, N.J. 1964.

\bibitem{G2014CVPDE} T. Giletti, Convergence to pulsating traveling waves with minimal speed in some KPP heterogeneous problems, \textit{Calc. Var. Partial Differential Equations } $\mathbf{51}$  (2014) 265-289.


\bibitem{H2008} F. Hamel, Qualitative properties of monostable pulsating fronts: exponential decay and
monotonicity, \textit{J. Math. Pures Appl. }  $\mathbf{89}$ (2008)  355-399.

\bibitem{HR2011} F. Hamel,  L. Roques, Uniqueness and stability properties of monostable pulsating
fronts, \textit{J. Eur. Math. Soc.} $\mathbf{13}$  (2011)  345-390.

\bibitem{PeterHess} P. Hess, Periodic-Parabolic Boundary Value Problems and Positivity,  Pitman Research Notes in Mathematics Series, vol. 247, Longman Sci. Tech., Harlow,  UK/New York, 1991.


\bibitem{LYZ2006JDE}X. Liang, Y.-F. Yi, X.-Q. Zhao, Spreading speeds and traveling waves for periodic evolution systems, \textit{J. Differential Equations} $\mathbf{231}$ (2006) 57-77.

 \bibitem{LZ2007} X. Liang,  X.-Q. Zhao, Asymptotic speeds of spread and traveling waves for monotone semiflows with applications, \textit{ Comm. Pure Appl. Math.} $\mathbf{60}$  (2007) 1-40.



\bibitem{M1978JMKU} H. Matano, Convergence of solutions of one-dimensional semilinear parabolic equations, \textit{J. Math. Kyoto Univ.} $\mathbf{18}$  (1978)  221-227.

\bibitem{P2011ARMA} P. Polacik, Threshold solutions and sharp transitions for nonautonomous parabolic equations on $\Bbb R^N$,  \textit{Arch. Ration. Mech. Anal.} $\mathbf{199}$  (2011)  69-97.

\bibitem{P2016JDDE} P.  Polacik, Threshold behavior and non-quasiconvergent solutions with localized initial data for bistable reaction-diffusion equations, \textit{ J. Dynam. Differential Equations} $\mathbf{28}$  (2016)  605-625.

\bibitem{PSIAM2017} P. Polacik, Planar propagating terraces and the asymptotic one-dimensional symmetry of solutions of semilinear parabolic equations, \textit{SIAM. J. Math. Anal. }  $\mathbf{49}$  (2017)  3716-3740.

\bibitem{P2017JFPTA} P. Polacik, Propagating terraces in a proof of the Gibbons conjecture and related results, \textit{J. Fixed Point Theory Appl.} $\mathbf{19}$  (2017), 113-128.

\bibitem{PMAMS} P. Polacik, Propagating terraces and the dynamics of front-like solutions of reaction-diffusion equations on $\mathbb{R}$, \textit{ Mem. Amer. Math. Soc.} to appear.


\bibitem{Shen1999JDE1} W.  Shen, Travelling waves in time almost periodic structures governed by bistable nonlinearities. I. Stability and uniqueness, \textit{J. Differential Equations } $\mathbf{159}$  (1999) 1-54.

\bibitem{Shen1999JDE2} W.  Shen, Travelling waves in time almost periodic structures governed by bistable nonlinearities. II. Existence, \textit{J. Differential Equations } $\mathbf{159}$  (1999) 55-101.

\bibitem{Shen2001JDE} W.  Shen, Dynamical systems and traveling waves in almost periodic structures, \textit{J. Differential Equations} $\mathbf{169}$  (2001)  493-548.

 \bibitem{Shen2011} W. Shen, Existence, uniqueness, and stability of generalized
traveling waves in time dependent monostable equations, \textit{J. Dynam. Differential Equations} $\mathbf{23}$  (2011) 1-44.


\bibitem{SS2017JDE1} W. Shen, Z. Shen, Transition fronts in time heterogeneous and random media of ignition type, \textit{J. Differential Equations} $\mathbf{262}$ (2017) 454-485.

\bibitem{SS2017DCDS} W.  Shen, Z. Shen, Transition fronts in nonlocal equations with time heterogeneous ignition nonlinearity, \textit{Discrete Contin. Dyn. Syst.} $\mathbf{37}$ (2017) 1013-1037.

\bibitem{SS2017JDE2} W. Shen, Z. Shen, Regularity and stability of transition fronts in nonlocal equations with time heterogeneous ignition nonlinearity, \textit{J. Differential Equations} $\mathbf{262}$  (2017)  3390-3430.

\bibitem{SS2017T} W. Shen, Z. Shen, Stability, uniqueness and recurrence of generalized traveling waves in time heterogeneous media of ignition type, \textit{Trans. Amer. Math. Soc.}  $\mathbf{369}$   (2017)  2573-2613.

\bibitem{SLW2012} W.-J. Sheng, W.-T. Li, Z.-C. Wang, Periodic pyramidal traveling fronts of bistable reaction-diffusion equations with  time-periodic nonlinearity, \textit{J. Differential Equations} $\mathbf{252}$ (2012) 2388-2424.

\bibitem{U1978}K. Uchiyama, The behavior of solutions of some nonlinear diffusion equations for large
time, \textit{J. Math. Kyoto Univ.}  $\mathbf{18}$  (1978) 453-508.

\bibitem{VVV1994} A.I. Volpert, V.A. Volpert, V.A. Volpert, Traveling Wave Solutions of Parabolic Systems. Translations of Mathematical Monographs 140, American Mathematical Society, Providence, 1994.

\bibitem{W2015} Z.-C. Wang, Cylindrically symmetric traveling fronts in periodic reaction-diffusion equation with bistable nonlinearity,  \textit{ Proc. Roy. Soc. Edinburgh Sect. A}  $\mathbf{145}$   (2015) 1053-1090.

\bibitem{WW2011} Z.-C. Wang, J. Wu, Periodic traveling curved fronts in reaction-diffusion equations with bistable time-periodic nonlinearity,  \textit{ J. Differential Equations} $\mathbf{250}$  (2011) 3196-3229.

\bibitem{ZXQ2003} X.-Q. Zhao, Dynamical Systems in Population Biology,  Springer-Verlag, New  York,  2003.

\bibitem{Z2006JAMS} A. Zlatos, Sharp transition between extinction and propagation of reaction, \textit{J. Amer.
Math. Soc.} $\mathbf{19}$ (2006)  251-263.

\end{thebibliography}
\end{document}